\documentclass[12pt,reqno]{amsart}
\usepackage{amsmath,amssymb,amsthm, amsfonts}
\usepackage{color}
\usepackage[breaklinks=true]{hyperref}
\usepackage{enumitem}

\theoremstyle{plain}
\newtheorem{theorem}{Theorem}[section]
\newtheorem{corollary}[theorem]{Corollary}
\newtheorem{lemma}[theorem]{Lemma}
\newtheorem{proposition}[theorem]{Proposition}

\newtheorem{conjecture}{Conjecture}[section]

\theoremstyle{definition}
\newtheorem{definition}[theorem]{Definition}

\newtheorem{remark}[theorem]{Remark}

\numberwithin{equation}{section}
\numberwithin{table}{section}
\allowdisplaybreaks

\newcommand{\BH}{\mathbf{B}(\mathcal{H})}
\newcommand{\BK}{\mathbf{B}(\mathcal{K})}
\newcommand{\hilbert}{\mathcal{H}}

\begin{document}
\title[ Discrete Schoenberg-Delsarte theory]{Complete Discrete Schoenberg-Delsarte theory for homogeneous spaces}

\author{Sujit Sakharam Damase}
\address[S.S.~Damase]{Department of Mathematics, Indian Institute of
Science, Bangalore 560012, India;}
\email{\tt sujits@iisc.ac.in}

\author{James Eldred Pascoe}
\address[J.E.~Pascoe]{Department of Mathematics, Drexel University, Philadelphia, PA 19104, United States;}
\email{\tt jep362@drexel.edu}
\thanks{We thank J C Bose Fellowship JCB/2021/000041 of SERB for supporting this research.}

\date{\today}



\begin{abstract}   
We develop a theory of partially defined complete positivity preservers, extending Schoenberg's classical characterization to functions defined only on discrete subsets or constrained domains. We frame the extension problem through the theory of completely positive maps on operator systems— we characterize general partially defined completely positive definite functions on general homogeneous spaces. We apply our interpolation to constrained packing problems and Delsarte theory, where one uses positive definite functions on homogeneous spaces to obtain bounds on various packing problems. We prove the specific positive definite function witnesses that a code is sharp for constrained angle codes  must be from polynomials.
\end{abstract}

\maketitle
\setcounter{tocdepth}{2}
\tableofcontents

\section{Introduction}
    A square matrix $[m_{ij}]_{i,j}$ with real entries is said to be {\bf symmetric} if 
    $m_{ij}=m_{ji}.$ A real symmetric matrix is said to be {\bf positive
    semidefinite} if 
    \[
        \sum_{i,j} c_ic_j m_{ij}\geq 0
    \]
    for any sequence of real numbers $c_1, \ldots, c_N.$

    Schoenberg \cite{schoenberg42} characterized {\bf positivity preservers}, continuous functions $f: [-1,1] \rightarrow \mathbb{R}$ such that given a positive semidefinite matrix $[m_{ij}]_{i,j}$ with entries in $[-1,1]$ one has that the matrix $[f(m_{ij})]_{ij}$ is again positive
    semidefinite as exactly $f$ of the form
        \[
            f(x) = \sum^{\infty}_{i=0} a_i x^i
        \]
    where $a_i\geq 0$ and $\sum^{\infty}_{i=0} a_i < \infty.$

    We analyze the theory of positive definite functions which are merely defined on $X\subseteq [-1,1].$ Moreover, we will extend the classical theory to the ``complete" case, that is to extend to the case of operator valued functions.

    A {\bf partially defined matrix} is a matrix with some entries left unspecified. For example, the matrix 
        \begin{equation}\label{simplematrix}
            \begin{bmatrix}
                 1 & ? & -1 \\
                 ? & 2 &  1 \\
                -1 & 1 &  ?
            \end{bmatrix}
        \end{equation}
    is a partially defined matrix. A partially defined matrix is 
    said to be {\bf positive semidefinite} if there exists a way to substitute values for the unspecified entries such that the matrix becomes positive semidefinite. That is, the partially defined matrix has a positive semidefinite {\bf completion}.
    For example, the above matrix \eqref{simplematrix} is positive semidefinite as 
        \begin{equation}\label{simplematrix2}
            \begin{bmatrix}
                 1 & 0 & -1 \\
                 0 & 2 &  1 \\
                -1 & 1 &  17
            \end{bmatrix}
        \end{equation}
    gives a positive semidefinite completion.

    We consider the preserver problem for partially defined positive semidefinite matrices. Given $X \subseteq \mathbb{R},$ we say
    $f: X \rightarrow \mathbb{R}$ is a {\bf partially defined positivity preserver} if for any partially defined positive semidefinite matrix $[m_{ij}]_{i,j}$ we have that $[f(m_{ij})]_{i,j}$ is again a partially defined positive semidefinite matrix, where $f$ of an unspecified entry is left unspecified. We note that if $f: [-1,1] \rightarrow \mathbb{R}$ is a positivity preserver, it must be a partial positivity preserver, as given any partially positive semidefinite matrix with entries in $[-1,1],$ any completion of it must have entries in $[-1,1]$ and thus one may apply $f$ to the completion to find a completion of $[f(m_{ij})]_{i,j}.$
    For example, if our partially defined matrix $[m_{ij}]_{i,j}$ is given by \eqref{simplematrix},
    and $f(x)=x^2,$ we have that
    $[f(m_{ij})]_{i,j}$ is given by
        \begin{equation}\label{simplematrix3}
            \begin{bmatrix}
                 1 & ? & 1 \\
                 ? & 4 &  1 \\
                1 & 1 &  ?
            \end{bmatrix}
        \end{equation}
    which can be completed to the positive semidefinite matrix
        \begin{equation}\label{simplematrix4}
            \begin{bmatrix}
                 1 & 0 & 1 \\
                 0 & 4 &  1 \\
                1 & 1 &  289
            \end{bmatrix}.
        \end{equation}


    We may now state a consequence of our general results proving a question asked in \cite{tohoku}.
    \begin{theorem}\label{mainresultpreview}
        If $\psi: \mathbb{Z} \rightarrow \mathbb{R}$ is a partially defined positivity preserver then
            \[
                \psi(x) = \sum^{\infty}_{i=0} a_i x^i
            \]
        where $a_n\geq 0$ and $\limsup a_n^{1/n} =0.$
    \end{theorem}
    In fact, we will see that one may replace $\mathbb{Z}$ by any set
    $X$ such that $\sup X = \infty.$ In the case where $X$ is bounded or even finite some more delicate considerations occur. We will also see that one may also replace the range $\mathbb{R}$ by the space of self-adjoint operators by phrasing the above definition in terms of the notion of complete positivity from $C^*$-algebras.

    \subsection{Outline}
        We begin with some basic universally appreciable motivation in Section \ref{motivation}. The further historical context, background and methods we adopt to tackle the above fairly pedestrian consideration are somewhat disparate and quite rich, including applications to spherical codes and packing problems. We will need to develop some basic understanding of group representation theory and positive definite functions of groups given in Section \ref{grouppd}, including the theory of homogeneous spaces, Class I representations, connections to other group theoretic properties such as Property (T) and some torment arising from discontinuous representations of topological groups.
        In Section \ref{discoscho}, we give a discretized proof of the classical Schoenberg's theorem using estimates for
        Krawtchouk polynomials.
        Furthermore, we require some notions from the theory of complete positivity to state and prove our results in full generality which we develop in Section \ref{completepositivity}, including the theory of Stinespring representation and Arveson extension. We discuss the theory of completely positive definite functions on homogeneous spaces in Section \ref{homogeneous}, with some consequences for the theory of Class I representations of $SO(\infty)$ over an index $1$ copy of itself. Applying Arveson extension type results, we then treat the theory of partially defined completely positive definite functions on homogeneous spaces in Section \ref{partialhomogeneous}. In Section \ref{constrainedanglesection}, we give applications to constrained angle spherical codes. In Section \ref{partialpoz}, we obtain a characterization of partially defined complete positivity preservers. Finally in Section \ref{conclusion}, we give some open problems, reflections and other considerations, including the problem of Delsarte hallucinations, where Delsarte's method returns an integer but no code of that size exists.
\section{Motivation} \label{motivation}
    Recall a Hilbert space $\mathcal{H}$ is a complete inner product space.

    We have the following gives a self-comprehension property of Hilbert spaces.
    \begin{theorem}[Riesz representation theorem]
        Let $\mathcal{H}$ be a Hilbert space.
        Let $\lambda:\mathcal{H}\rightarrow \mathbb{C}$ be a continuous linear map.
        There is $\hat{\lambda}\in \mathcal{H}$ such that
            $\lambda(f)=\langle f,\hat{\lambda}\rangle.$
    \end{theorem}

    Assuming that $\mathcal{H}$ is a space of functions and that the evaluation maps $\lambda_x(f)=f(x)$ are continuous gives us the notion of a reproducing kernel Hilbert space.
    \begin{definition}[Reproducing kernel Hilbert spaces]
         Suppose a $\mathcal{H}$ is a space of functions and that the evaluation maps $\lambda_x(f)=f(x)$ are continuous. We call $\mathcal{H}$ a reproducing kernel Hilbert space.

         We call the elements $k_x$ such that $f(x)=\langle f, k_x\rangle$ the kernel functions.
    \end{definition}

    Finally, there is the basic representer theorem.
    \begin{theorem}[Aronszajn-Moore-Kolmogorov type theorem] \label{AMK}
    Let $X$ be a set. Let $K: X \times X\rightarrow \mathbb{C}$
    such that $[K(x_i,x_j)]_{i,j}$ is positive semidefinite for 
    every finite set $x_1,\ldots, x_d.$
    There is a map $\iota: X \rightarrow \mathcal{H}$
    where $\mathcal{H}$ is a Hilbert space and 
        $$\langle \iota_x, \iota_y \rangle = K(x,y).$$
    \end{theorem}
    \subsection{Statistics and numerics}
        A covariance matrix describes the statistical relationships between various data, and thus are ubiquitous in the modern era. For large covariance matrices, there is a temptation to sparsify system by replacing small entries of the covariance matrix with $0.$ However tempting that might be, usually the resulting matrix is not positive semi-definite. Essentially the reason is that power series with positive coefficients cannot have zeros for positive inputs. However, by working either on preserver problems for fixed dimension or only thresholding small negative entries, one obtains a feasible framework for soft thresholding. Rajaratnam and Guillot \cite{guillot2012retaining}
        describe the theoretical foundations of soft thresholding in terms of Schoenberg's theory and themes. Khare and Tao \cite{khare-tao} constructed preservers in fixed dimensions using Schur polynomials. Additionally, matrix completion problems arise naturally from partially observed covariance matrices, see e.g. \cite{ben2012positive}.
        \subsubsection{Floating point arithmetic}
        A natural question was whether Schoenberg's theorem on positivity perservers held in floating point arithmetic. (For example, in the natural context where you apply the aforementioned soft-thresholding.) That is, by discretizing the domain, it was not clear how many new positivity preservers would appear. Within the context of partially defined positivity preservers, our Theorem \ref{mainresultpreview} establishes nothing new arises.
    \subsection{Conservative interpolation}
        General naive approaches to interpolation suffer from huge amounts of instability, such as Runge and Gibb's type phenomena. One sees in classical estimates for Lagrange interpolation \cite{szego, gautschi1974norm} that interpolation is indeed highly unstable. An approach is to choose from a normal type family instead, a set of functions with sufficient compactness properties to guarantee stability.
        
        \subsubsection{Nevanlinna-Pick interpolation}
            For example. Nevanlinna-Pick interpolation theorem
            describes when there exists an analytic function from the unit disk to the closed unit disk taking prescribed values at particular points.
            \begin{theorem}[Pick \cite{pick}, Nevanlinna \cite{nevanlinna}]
                Let $z_1, \ldots, z_n \in \mathbb{D}.$ Let $\lambda_1, \ldots, \lambda_n \in \mathbb{C}.$ There exists an
                analytic function $\varphi: \mathbb{D}\rightarrow \mathbb{D}$
                such that $\varphi(z_i)=\lambda_i$ if and only if
                the matrix
                    \[
                        \begin{bmatrix}
                            \frac{1-\lambda_i\overline{\lambda_j}}{1-z_i\overline{z_j}}
                        \end{bmatrix}_{i,j}
                    \]
                is positive semidefinite.
            \end{theorem}
    
            One may reinterpret Theorem \ref{mainresultpreview} as solving a kind of interpolation problem for positive definite functions, and we will later solve the problem in general for completely positive definite functions via partially defined completely positive definite functions.

            \subsection{Sphere packing and Schoenberg's theorem}
            The sphere packing problem concerns with the densest packing of spheres into Euclidean space $\mathbb{R}^d$. That is, what fraction of $\mathbb{R}^d$ can be covered by the congruent balls that do not intersect except along their boundaries? There has been a lot of work in this direction such as \cite{cohn-elkies}, \cite{Cohn-woo}, \cite{Cohn-Zhao},  and of course the celebrated \cite{viazovska2017sphere}, \cite{cohn2017sphere}.

            Linear programming bounds are the most powerful known technique for producing upper bounds in such problems. In particular, \cite{K-L} uses this technique to prove the best bounds known for sphere packing density in high dimensions. However, \cite{K-L} does not study sphere packing directly, but rather passes through the intermediate problem of spherical codes.
            Let $\Delta_{\mathbb{R}^n}$ denote the maximal packing density.

            A spherical code $\mathcal{C}$ in dimension $n$ with a minimum angle $\theta$ is a set of points on the unit sphere $S^{n-1}$ so that no two are closer than angle $\theta$ to each other along the the great circle connecting them. That is, $\langle x,y\rangle \leq \cos{\theta}$ for all pairs of distinct points $x,y \in \mathcal{C}.$ Let $A(n, \theta)$ denote the greatest size of such a spherical code.

            In \cite{Delsarte}, Delsarte used linear programming technique to give upper bounds for packing density in the setting of error-correcting codes. Delsarte, Goethals, and Siedel \cite{DGS} and Kabitiansky and Levenshtein \cite{K-L} independently formulated a linear program for proving upper bounds on $A(n,\theta)$. Using this technique, Kabitiansky and Levenshtein \cite{K-L} obtained an upper bound on $A(n, \theta)$ and then use the inequality 
            \[
            \Delta_{\mathbb{R}^{n}} \leq \sin^{n}{(\theta/2)} A(n+1, \theta).
            \]
            This inequality was further improved by Cohn and Zhao \cite{Cohn-Zhao}. They proved that, for all $\theta\in [\pi/3, \pi],$
            \[
            \Delta_{\mathbb{R}^{n}} \leq \sin^{n}{(\theta/2)} A(n, \theta).
            \]
             Now, we establish the connection between the spherical code bounds and Schoenberg's theorem. For a positive definite function $g: [-1,1]\to \mathbb{R},$ define $\Bar{g}$ to be its \textbf{average} by
             \[
             \Bar{g} = \dfrac{\int_{-1}^{1} g(t) (1-t^{2})^{(n-3)/2} dt}{\int_{-1}^{1} (1-t^{2})^{(n-3)/2}dt} = \displaystyle \mathop{\mathbb{E}}_{x,y\in S^{n-1}} g(\langle x,y \rangle).
             \]
             That is, $\Bar{g}$ is the expectation of $g(\langle x,y\rangle)$ with $x$ and $y$ chosen independently and uniformly at random from $S^{n-1}.$
             A function $g: [-1,1]\to \mathbb{R}$ is said to be \textbf{positive definite function} on the sphere $S^{n-1}$ if, for all $N$ and all $x_{1},\dots, x_{N}$, the matrix $(f(\langle x_{i}, x_{j}\rangle))_{1\leq i,j\leq N}$ is positive semidefinite.
             By Schoenberg's theorem, if $g: [-1,1]\to \mathbb{R}$ is a continuous positive definite function on the sphere $S^{n-1}$ then it can be written as the nonnegative linear combination of the Gegenbauer polynomials $C_{k}^{(n/2 - 1)}$ for $k\in \mathbb{Z}_{\geq 0}.$ The family of Gegenbauer polynomials or ultraspherical polynomials $C_{k}^{(\alpha)}$ occurs as orthogonal polynomials with respect to the measure $(1-t^2)^{\alpha - 1/2} dt$ on $[-1,1]$. For $\alpha = n/2 -1,$ this measure arises naturally (up to scaling) as the orthogonal projection of the surface measure from $S^{n-1}$ onto $x$-axis. 

             \begin{theorem}[Delsarte–-Goethals–-Seidel \cite{DGS}, Kabatiansky–-Levenshtein \cite{K-L}, Cohn--Zhao \cite{Cohn-Zhao}]\label{delsartegs}
                Let $g: [-1,1]\to \mathbb{R}$ be a continuous, positive definite kernel on $S^{n-1}$ with the following properties:
                \begin{itemize}
                    \item $g(t)\leq 0,$ if $t\in [-1, \cos{\theta}];$
                    \item $\Bar{g}>0$.
                \end{itemize}
                Then
                \[
                A(n, \theta) \leq \frac{g(1)}{\Bar{g}}.
                \]
             \end{theorem}
             We will give a proof of this theorem following the approach of \cite{Cohn-Zhao} for illustrative purposes.
             \begin{proof}
                 Let $\mathcal{C}$ be any spherical code in $S^{n-1}$ with minimal angle at least $\theta,$ let $\mu$ be the surface measure on $S^{n-1}$, normalized to have total measure $1$, and let
                 \[
                 \nu = \sum_{x\in \mathcal{C}} \delta_{x} - |\mathcal{C}| \mu,
                 \]
                 where $\delta_x$ denotes the Dirac measure at point $x$, and $|\mathcal{C}|$ denotes the size of the code. Therefore,
                 \[
                 \iint_{S^{n-1}} g(\langle x,y \rangle) \varphi(x) \varphi(y) d\nu(x) d\nu(y) \geq 0,\ \text{for all}\ \varphi\in L^{1}(S^{n-1}).
                 \]
                 Specializing this to $\varphi(x)$ being equal to the constant function $1$ we get,
                 \[
                 -|\mathcal{C}|^{2} \Bar{g} + \sum_{x,y \in \mathcal{C}} g(\langle x,y \rangle) \geq 0
                 \]
                 By definition, all distinct points in the code $\mathcal{C}$ are at least at the distance $\theta$ and hence $g(\langle x,y\rangle)\leq 0,$ for all $x,y \in \mathcal{C}.$ Therefore,
                 \[
                 \sum_{x,y \in \mathcal{C}} g(\langle x,y \rangle) \leq \sum_{x \in \mathcal{C}} g(\langle x,x \rangle) = g(1) |\mathcal{C}|.
                 \]
Substituting this in the above inequality gives
\[
 |\mathcal{C}| \leq \frac{g(1)}{\Bar{g}}.
\]

             \end{proof}
            Concretely, one may view the positive definite function $g$ as inducing a map from $S^{n-1}$ to some high or even infinite dimensional sphere such that any code with minimal angle acute are mapped to some code with a strictly obtuse angle, (cosine negative) for which there is an absolute dimension independent bound. (The sum of the entries of the Gramian $G$ must be positive, as if all the nondiagonal entries are less than $-\varepsilon,$ then there is a limit on the size of $G.$)
            \subsection{Lov\`{a}sz theta function and Schoenberg's theorem}

            We state some lower bounds on measurable chromatic numbers based on semidefinite programming and provide proofs if we wish to use in the forthcoming section.
            \begin{definition}
                The \textbf{chromatic number} of $\mathbb{R}^{n},$ denoted by, $\chi(\mathbb{R}^{n}),$ is the minimum number of colours needed to colour each point of $\mathbb{R}^{n}$ in such a way that points at distance $1$ from each other get different colours.
            \end{definition}
            That is, $\chi(\mathbb{R}^{n})$ is the chromatic number of the graph with vertex set $\mathbb{R}^{n}$, in which two vertices are adjacent if they are at a unit distance.
            \begin{definition}
                The \textbf{measurable chromatic number} of $\mathbb{R}^{n},$ denoted by $\chi_{m}(\mathbb{R}^{n}),$ is the smallest number $m$ such that $\mathbb{R}^{n}$ can be partitioned into $m$ Lebesgue measurable subsets $C_{1},\cdots, C_{m}$ with no two points in $C_{i}$ lie at the distance $1.$ 
            \end{definition}
            Note that, the measurable chromatic number is at least the chromatic number. We can define a similar notion on the the sphere $S^{n-1}.$ For $t\in (-1,1),$ consider the graph $G(n,t)$ whose vertices are the points of $S^{n-1}$ and in which two points $x,y$ are adjacent if $\langle x,y \rangle = t.$ The (measurable) chromatic number of $G(n,t)$ are defined as in the case of $\mathbb{R}^{n},$ and denoted by $(\chi_{m}(G(n,t)))$ $\chi_{m}(G(n,t)).$

            Lov\`{a}sz defined the theta function for a finite graph $G=(V,E)$ as follows:
            \begin{align*}
                \vartheta(G) = \max \bigg\{ 
                \sum_{x \in V} \sum_{y \in V} K(x,y) : & K \in \mathbb{R}^{V\times V}\ \text{is positive semidefinite},\\
                & \sum_{x\in V} K(x,x)= 1,\ K(x,y) = 0\ \text{if}\ xy\in E
                \bigg\}.
            \end{align*}
            The above notion can be extended \cite{gafa} to a compact metric space $(V,d)$ equipped with a Haar measure $\mu.$ Let $D$ be the range of the distance function $d.$ Then we can define the graph $G= G(V, D)$ with vertex set $V$ and edge set $E= \{xy : d(x,y) \in D\}.$ Now, we define the Lov\`{a}sz theta function to this graph $G$
            \begin{align*}
                \vartheta(G) = \max \bigg\{ 
                \iint_{V\times V} & K(x,y) d\mu(x) d\mu(y) :  K \in \mathcal{P}(V\times V),\\
                & \int_{V} K(x,x) d\mu(x) = 1,\ K(x,y) = 0\ \text{if}\ d(x,y)\in E
                \bigg\}.
            \end{align*}
            See \cite{gafa} for the following result.
            \begin{theorem}
                The theta function is an upper bound for the stability number, i.e.
                \[
                \vartheta(G) \geq \alpha(G),
                \]
                where $\alpha(G)$ is the supremum over the measure of all subsets $C$ of $V$ with no two elements in $C$ are adjacent.
            \end{theorem}
            Now, we move onto the homogeneous space $V = S^{n-1}$ endowed with the Euclidean metric on $\mathbb{R}^{n},$ and hence $D$ is a singleton. If $D=\{\delta\}$ and $\delta^{2} = 2 - 2t,$ so that $d(x,y) = \delta$ if and only if $\langle x,y \rangle = t.$ We denote $G(V, D)$ by $G(n,t).$ Therefore, $\Gamma = O(\mathbb{R}^{n})$ acts on the sphere $S^{n-1}$ and this action does not change the value of the theta function, so we get the following characterization
             \begin{align*}
                \vartheta(G) = \max \bigg\{ 
                &\iint_{S^{n-1}}  K(x,y) d\omega(x) d\omega(y) :  K \in \mathcal{P}(V)\ \text{is $\Gamma$-invariant},\\
                & K(x,x)= 1/\omega_{n} \ \text{for all}\ x,\ K(x,y) = 0\ \text{if}\ \langle x,y\rangle = t
                \bigg\},
            \end{align*}
            where $w_{n}$ is the surface area of $S^{n-1}$ and $d\omega$ is the Haar measure on $S^{n-1}.$ By Schoenberg's theorem, one can convert the semidefinite programming problem in infinitely variables with infintely many constraints into the following linear programming problem with optimization variables $a_{k}$:
             \begin{align*}
                \vartheta(G) = \max \bigg\{ \omega_{n}^{2} a_{0} :  a_{k} \geq 0,\ & \text{for}\ k=0,1,\cdots,\\
                & \sum_{k=0}^{\infty}a_{k} = 1/\omega_{n},\ a_{0} +  \sum_{k = 1}^{\infty}a_{k} P_{k}^{(n-3)/2} = 0
                \bigg\}.
            \end{align*}
            \begin{theorem}[\cite{gafa}]
               Let $m(t)$ be the minimum of $P_{k}^{(n-3)/2}$ for $k\in \mathbb{Z}_{\geq 0}.$ Then the optimum value of the above equation is equal to
               \[
               \vartheta(G) = \omega_{n} \frac{m(t)}{m(t)-1}.
               \]
            \end{theorem}
           Our Theorem \ref{firstsharptheorem} recovers the previous in greater generality. Specifically, it makes clear why we can choose the minimum and not the infimum.
            
            Let $\Bar{G}$ be the \textbf{complement} if graph $G,$ that is, the graph with the same vertex set as $G$ but the two vertices in $\Bar{G}$ are adjacent if and only if they are nonadjacent in $G.$ Then there is a different definition of $\vartheta$ for $G=(V,E)$ (which is equivalent for finite graphs) as follows:
             \begin{align*}
                \vartheta(\Bar{G}) = \min \{ \lambda :  K \in \mathbb{R}^{V\times V}\ \text{is PSD},\ & K(x,x) = \lambda-1\ \text{for all}\ x\in V, \\
                & K(x,y) = -1\ \text{if}\ xy\in E\}.
            \end{align*}
            and its generalization, denoted by $\Bar{\vartheta}(G(n,t))$ is
             \begin{align*}
                \Bar{\vartheta}(G(n,t)) = \min \{ \lambda :& K \in \mathcal{P}(V)\ \text{is $\Gamma$-invariant},\\
                & K(x,x) =\lambda - 1,\ \text{for all}\ x\in S^{n-1},\\
                & K(x,y) = -1\ \text{if}\ \langle x,y\rangle = t\}.
            \end{align*}
            This generalization of the theta function is closely related to the Delsarte's LP bounds \cite{Delsarte}. This connection has been observed in several places, for example, see \cite{gafa}.

            Consider the graph on the unit sphere where two distinct points are adjacent
whenever their inner product lies in the open interval $[-1, t].$ We denote this graph
by $G(n, [-1, t])$. Stable sets in the complement of this graph are finite and consist
of points on the unit sphere with minimal angular distance $\arccos(t).$ Applying the above generalization to this graph we get
 \begin{align*}
                \Bar{\vartheta}(G(n,[-1,t])) = \min \{ \lambda :& K \in \mathcal{P}(V)\ \text{is $\Gamma$-invariant},\\
                & K(x,x) =\lambda - 1,\ \text{for all}\ x\in S^{n-1},\\
                & K(x,y) = -1\ \text{if}\ \langle x,y\rangle \in [-1, t]\}.
            \end{align*}
Now, we can write the \cite{gafa}[Proposition 10.1].
\begin{theorem}
    Let $C\subseteq S^{n-1}$ be a spherical code such that 
    \[
    \langle x,y \rangle \in [-1,t],\ \text{for all}\ x,y \in C.
    \]
    Then the cardinality of $C$ is bounded above by $\Bar{\vartheta}(G(n, [-1,t])).$
\end{theorem}
{

We discuss the theory where we replace $[-1, t]$ with a general finite set in Section \ref{constrainedanglesection}. For example, if our finite set does not contain $-1$ there there is a polynomial Delsarte witness by Theorem \ref{firstsharptheorem}.      
            
    \subsection{Cone preservers}
        Finally, we note the relation to the problem posed in \cite{tohoku} related to cone preservers: is there a proof of Schoenberg's theorem for discrete sets? Our analogue, in terms of partially defined positive definiteness, gives a reasonable resolution. Whether one can drop the matrix completion aspect of our definition to obtain something similar is a soft open question about weakest assumptions-- however, we are skeptical. However Section \ref{discoscho} also accomplishes the other aim of \cite{tohoku} to give an ``analysis free" proof of Schoenberg's theorem.
\section{Group representation theory and positive definite functions on groups} \label{grouppd} 
    \subsection{Homogeneous spaces}
    Let $G$ be a compact group. We say, $X$ is a \textbf{2-point homogeneous $G$-space} if the following three properties hold:
    \begin{enumerate}
        \item There is a $G$-action on $X$;
        \item $X$ is a metric space with a distance function $d$ defined on it;
        \item $d$ is strongly invariant under the $G$-action: for any $x_{1}, x_{2}, y_{1}, y_{2}\in X$, $d(x_{1},y_{2}) = d(x_{2},y_{2})$ if and only if there is an element $g\in G$ such that $g(x_{1})= x_{2}$ and $g(y_{1})= y_{2}$.
    \end{enumerate}
    Assumption (3) implies that $G$ acts transitively on $X$. Moreover, $G$ is as close to being doubly transitive as it can be without changing the distances $d(x,y)$ between the points of $X$.  This induces a unique invariant measure $\mu$ on $X$. We assume that $\mu$ is normalized so that $\mu(X)= 1$.
    
    Let $H$ be a subgroup of $G$ that fixes an element $x_{0}\in X$. By Orbit-Stabilizer theorem, $X$ can be identified with the space $G/H$ of left cosets $gH$. These assumptions imply that if $G$ is infinite then $X$ is a compact Riemannian symmetric space of rank one \cite{wang}. Furthermore, this is a particular instance of a Gelfand pair $(G, H)$.   

    The main examples here is when $G$ is the special orthogonal group $SO(n)$ acting on $X = S^{n-1}$ with the natural action. Choosing $x_{0}$ to be the first basis vector $e_{1} = (1,0,\cdots,0)$ gives us that $H = SO(n-1)$. And in the \textit{discrete case} we consider the \textit{Hamming cube} $\{0,1\}^{n}$ with the Hamming metric 
    \[
    d(x,y) = |\{i : x_{i}\neq y_{i}\}|
    \]
    with the isometry group being the wreath product $G= S_{n} \rtimes (\mathbb{Z}/2\mathbb{Z})^{n}.$ The base point can be taken to be $x_{0} = (0,\cdots, 0)$ to get the stabilizer $H = S_{n}.$
    \subsection{Group representation theory}
        
        \subsubsection{Representation theory of compact groups}
Consider the circle $\mathbb{R}/\mathbb{Z} $. For each $n \in \mathbb{Z}$, define
\[
e_n \colon \mathbb{R}/\mathbb{Z} \to \mathbb{C}\ \text{by}\ e_n(t) = e^{2\pi i n t}.
\]
The theory of Fourier Series tells us that $\{e_n\}$ is an orthonormal basis of $L^{2}(\mathbb{R}/\mathbb{Z})$, which means we have the direct sum decomposition into irreducibles
\[
L^{2}(\mathbb{R}/\mathbb{Z}) = \bigoplus_{n \in \mathbb{Z}} \mathbb{C} e_n .
\]
Furthermore, each subspace $\mathbb{C} e_n$ is an invariant subspace for the regular representation. Hence, we have a decomposition of the regular representation into a direct sum of irreducible representations. In addition, it is not difficult to see that every irreducible representation of $S^{1}$ occurs exactly once in this representation.

More generally, it is not difficult to show that every unitary representation of~$S^{1}$ is a direct sum of $1$-dimensional representations.

The following theorem generalizes this to any compact group.   
\begin{theorem}[Peter--Weyl Theorem] \label{PeterWeyl}
Let $G$ be a compact group with a (unique) normalized $G$-invariant Haar measure $\mu$. Then:
	\begin{enumerate}
	\item \label{PeterWeyl-sum}
	Every unitary representation of~$G$ is\/ \textup(isomorphic to\/\textup) a direct sum of irreducible representations.
	\item \label{PeterWeyl-fd}
	Every irreducible representation of~$G$ is finite-dimensional. 
	\item \label{PeterWeyl-countable}
	$\widehat G$ is countable.
	\item \label{PeterWeyl-regrep}
	For the particular case of the regular representation $\bigl(\pi_{reg}, L^{2}(G) \bigr)$, we have
		\[
        \pi_{reg} \cong \bigoplus_{(\pi,\mathcal{H}) \in \widehat G} (\dim \mathcal{H}) \cdot \pi ,
        \]
	where $k \cdot \pi$ denotes the direct sum $\pi \oplus \cdots \oplus \pi$ of $k$~copies of\/~$\pi$. That is, the multiplicity of each irreducible representation is equal to its dimension.
	\end{enumerate}
\end{theorem}
 \subsubsection{Ultraspherical polynomials}
 In this section, we look at the ultraspherical polynomials from different point of views which will be useful in upcoming sections.
 Ultraspherical/ Gegenbauer polynomials $\{P_{k}^{(n)}\}_{k\in \mathbb{Z}_{\geq 0}}$ for the $n$-dimensional sphere $S^{n-1}$ can be defined as a family of \textit{orthogonal polynomials} with respect to the projection of the surface measure $(1-t^{2})^{(n-3)/2} dt$ on $[-1,1]$. Equivalently, $P_{k}^{(n)}$ is orthogonal to all polynomials of degree less than $k$ with respect to this measure. 

 Orthogonal polynomials enjoy many wonderful properties. We note here some of these properties starting with the three-term recurrence relation.
 \begin{theorem}
     Let $\{p_{n}\}$ be a family of orthogonal polynomials. Then there exists constants $A_{n},\ B_{n}$ and $C_{n}$ with $A_{n}, C_{n}>0$ such that the following relation holds:
     \[
     p_{n}(x) = (A_{n}x+B_{n})p_{n-1}(x) - C_{n} p_{n-2}(x), \quad n=2,3,4,\cdots.
     \]
     Moreover, if the highest coefficient of  $p_{n}(x)$ is denoted by $k_n,$ we have
     \[
     A_{n} = \frac{k_{n}}{k_{n-1}}, \qquad C_{n} = \frac{A_{n}}{A_{n-1}} = \frac{k_{n} k_{n-2}}{k_{n-1}^{2}}.
     \]
 \end{theorem}
 Next, we give a very useful formula for orthogonal polynomials which gets used in the sphere packing bounds as well, for instance, see \cite{K-L}.
 \begin{theorem}[Christoffel--Darboux formula]
  With above notation, we have
  \[
  \sum_{i=0}^{n} p_{n}(x)p_{n}(y) = \frac{k_{n}}{k_{n+1}} \dfrac{p_{n+1}(x)p_{n}(y) - p_{n}(x)p_{n+1}(y)}{x-y}
  \]
 \end{theorem}
It is amazing that an arcane property like orthogonality gives us a lot of information about the zero set of the polynomials that is, all roots of orthogonal polynomials are real, simple, and they interlace. More precisely,

\begin{theorem}\label{interlacing}
    Let $\{p_{n}\}_{n \in \mathbb{Z}\geq 0}$ be a collection of orthogonal polynomials with respect to the measure $\mu$ in the interval $[a,b].$ Then
    \begin{enumerate}
        \item The zeros of $p_{n}$ are real, distinct, and are contained in $(a,b).$
        \item Let $x_{1}< x_{2}<\cdots < x_{n}$ be the zeros of $p_{n}$. Then each interval $[x_{i}, x_{i+1}],\ i=0,1,2,\cdots, n,$ contains exactly one zero of $p_{n+1}.$
        \item Between two zeros of $p_{n}$ there is at least one zero on $p_{m},\ m>n.$
    \end{enumerate}
\end{theorem}
 Here, we   take a detour to assimilate some finer properties of orthogonal polynomials such as more refined results on location of zeros of these polynomials, Darboux formula for asymptotics proving \textit{concentration of measure}-type results. We start with the following density results for zero set of orthogonal polynomials.
 \begin{theorem} \label{density}
     Let $\{p_{n}\}_{n\in \mathbb{Z}\geq 0}$ be a set of orthogonal polynomials defined on a finite interval $[a,b]$ with respect to the measure $\mu.$ Let $[c,d]$ be a subinterval of $[a,b]$ with the property that $\int_{[c,d]} d\mu(x) > 0.$ Then there exists an integer $n_{0}$ such that, for all $n\geq n_{0}$ we have $Z(p_{n}(x)) \cap [c,d] \neq \emptyset$, where $Z(p_{n}(x))$ denotes the zero set of the polynomials $p_{n}$
 \end{theorem}

\begin{remark}
    One immediately sees that, for any $\epsilon>0$ and an interval $[c,d] \subset [-1,1]$ of length $\epsilon,$ the integral $\int_{[c,d]} d\mu(x) > 0,$ where $d\mu(x) = (1-x^{2})^{(n-3)/2} dx.$ Thus, the zero set of Gegenbauer polynomials is dense in $[-1,1].$ 
\end{remark}

Now, we get some meaningful bounds on the distribution of zeros of the polynomials at hand-Jacobi polynomials with parameters $\alpha$ and $\beta$:
\begin{theorem}
    Let $\alpha, \beta > -1,$ and let $0< \theta_{1}< \theta_{2}<\cdots < \theta_{n} <\pi$ be the zeros of $P_{n}^{(\alpha, \beta)}(\cos{\theta}).$ Then
    \[
    \theta_{j} = \frac{1}{n} (\pi j + \operatorname{O}(1)),
    \]
    where $O(1)$ is uniformly bounded for all $j=1,2,\cdots, n$, and $n\in \mathbb{N}.$
\end{theorem}
Finally, we come to the Darboux formula for asymptotics for the Jacobi polynomials.
\begin{theorem}\label{darbouxest}
    Let $\alpha$ and $\beta$ be arbitrary real numbers. Then for $0< \theta < \pi,$
    \[
    P_{n}^{(\alpha, \beta)}(\cos{\theta}) = n^{-1/2} k(\theta) \cos{(N\theta + \gamma)} + O(n^{-3/2}),
    \]
    where 
    \[
    k(\theta) = \frac{1}{\sqrt{\pi}} \sin{\bigg(\frac{\theta}{2}\bigg)}^{-\alpha-1/2} \cos{\bigg(\frac{\theta}{2}\bigg)}^{-\alpha-1/2}, \qquad N= n + \frac{\alpha + \beta + 1}{2},
    \]
    and $ \gamma = (-\alpha + 1/2) \pi/2.$ The bound for the error term holds umformly in the interval $[\epsilon, \pi-\epsilon],$ for every $\epsilon>0.$
\end{theorem}
Using Stirling approximation for Gamma function, we can give asymptotics for $P_{k}^{(\alpha, \beta)}(1)$ as follows:
\begin{theorem}[Szeg\H{o} \cite{szego}]
    For the Jacobi polynomials $P_{k}^{(\alpha, \beta)}(x)$
    \[
    P_{k}^{(\alpha, \beta)}(1) = \binom{k +\alpha}{k} = \dfrac{\Gamma(2\alpha + k)}{\Gamma(2\alpha) k!} \sim \frac{1}{\Gamma(2\alpha)} k^{2\alpha -1}.
    \]
\end{theorem}
 We shall use spherical harmonics to explain further properties of the ultraspherical polynomials. Recall that as a representation of $O(n),$ we can decompose $L^{2}(S^{n-1})$ as 
 \[
 L^{2}(S^{n-1}) = \overline{\bigoplus_{k\geq 0} W_{k}}^{L^2(S^{n-1})},
 \]
 where $W_k$ be the irreducible representation of $O(n)$ which consists of degree $k$ spherical harmonics. 

For $x\in S^{n-1},$ consider the evaluation functional taking $f\in W_k$ to $f(x)$. By Hilbert space duality, there exists a unique reproducing kernel $w_{k,x} \in W_{k}$, that is,
\[
f(x) = \langle f, w_{k,x}\rangle, \quad \text{for all}\ f\in W_k.
\]
\begin{remark}{Invariance of reproducing kernel under $O(n)$}
    For any orthogonal matrix $T$ with $Tx = x$, we have $Tw_{k,x} = w_{k,x}$. This is because: for all $f\in W_{k},$
    \[
    \langle f, w_{k,x}\rangle = f(x) = f(Tx) = (T^{*}f)(x) = \langle T^{*}f, w_{k,x}\rangle = \langle f, Tw_{k,x}\rangle
    \]
    holds and hence $ w_{k,x} = Tw_{k,x}.$
\end{remark}
The preceding remark can be reframed in the language of inner product. That is, $w_{k,x}(y)$ can depend only on the distance between $x$ and $y$, and therefore, it must be the function of $\langle x, y\rangle$ alone. Define $P_{k}^{(n)}$ by
\[
 w_{k,x}(y) = P_{k}^{(n)}(\langle x,y\rangle).
\]
Using RKHS property, $P_{k}^{(n)}(\langle x,y\rangle) = \langle  w_{k,x},  w_{k,y}\rangle.$
    \subsection{Positive definite functions on groups}
    There are many notions of positive-definiteness for functions on several class of spaces. But all of these notions essentially come from the notion that we saw in the Theorem \ref{AMK} combined with a group invariance. We start with a purely algebraic definition of positive definiteness.
    \begin{definition}[Positive Definite Function]
        Let $G$ be any group. A complex-valued function $\varphi$ on $G$ is said to be \textbf{positive definite (p.d.)} if for any $n$ and $g_{1},\dots, g_{n} $ in $G$, the matrix $(\varphi(g_{i}^{-1}g_{j}))$ is positive semidefinite.
        \end{definition}
    In other words, the kernel $K:G\times G \to \mathbb{C}$ given by $K(h,g) = \varphi(g^{-1}h)$ is positive definite in the sense of Theorem \ref{AMK}. Strictly speaking, this is a $G$-invariant positive definite kernel on the space $G$.
    In the 1930s, S. Bochner \cite{Bochner} classified all of the positive definite functions on $\mathbb{R}$ in terms of positive finite measures on the $\mathbb{R}$. This result was extended in 1940 simultaneously by Weil, Povzner, and Raikov to classify the positive definite functions on any locally compact abelian group. Subsequently, Godement \cite{godement} extended it to the case of Gelfand pairs  building on the ideas of Gelfand and Raikov \cite{gelfandraikov}. Around the same time, in his 1942 paper \cite{schoenberg42} Schoenberg was interested in understanding the positive definite functions of the form $f \circ \cos : [-1, 1] \to \mathbb{R}$ on a unit sphere $S^{r-1} \subset \mathbb{R}^r$, where $r \geq 2$.

Note that, for $r \ge 2$, we have $\cos(d(x_j, x_k)) = \langle x_j, x_k \rangle$, so $\cos[(d(x_j, x_k))_{j,k}]$ always yields Gram matrices. Hence, $f[-]$ would preserve positivity on a set of positive semi-definite matrices. It was this class of functions that Schoenberg characterized:

\begin{theorem}[Schoenberg \cite{schoenberg42}] \label{schoclassic}
Suppose $f : [-1, 1] \to \mathbb{R}$ is continuous, and $r \ge 2$ is an integer. Then the following are equivalent:
\begin{enumerate}
    \item $(f \circ \cos)$ is positive definite on $S^{r-1}$.
    \item The function $f(x) = \sum_{k=0}^\infty c_k C_k^{((r-2)/2)}(x)$, where $c_k \ge 0, \forall k$, and $\{C_k^{((r-2)/2)}(x) : k \ge 0\}$ denote the Gegenbauer family of orthogonal polynomials.
\end{enumerate}
\end{theorem}
As we shall see, this result can be seen as characterizing the $SO(n)$-invariant kernels on the sphere $S^{n-1},$ which enables us to use the harmonic analysis on the $2$-point homogeneous spaces.
     
    \subsection{Extension problems for positive definite functions}
        The extension problem for partially defined positive definite functions has been considered 
        extensively in other regimes based around convolution going back to Raikov \cite{Raikov1940}, Krein \cite{Krein1949}, Devinatz \cite{devinatz1959extensions}, Rudin \cite{Rudin1963}, McMullen \cite{McMullen1972},
        Bakonyi-Timotin \cite{BakonyiTimotin2011}, Berezanskii-Gorbachuk \cite{BerezanskiiGorbachuk1965}, Nussbaum
        \cite{Nussbaum1975}, Jorgensen-Niedzialomski \cite{jorgensen2015extension}.
        See Jorgensen \cite{jorgensen2016extensions} for a general overview of the modern theory.
        Essentially one takes a partially defined
        kernel on a set $\Gamma$ induced by a function on $\Gamma -\Gamma$ (that is a set convoluted with itself) and attempts to extend to the whole space.
        The analysis is in general somewhat delicate and there are many deeply interesting pathologies and obstructions to extension.
        The approach from the point of view of complete positivity gives a fundamentally different framework which is somewhat immunized to pathology as completely positive maps have a significant amount of rigid algebraic structure, e.g. the Stinespring representation, Arveson extension theorem and so on.
        
        For our particular practical problems, our notion of interpolation is natural, and may provide facility in other regimes.
    \subsection{Representations of Class I}
    Let $G$ be a locally compact group with left Haar measure $dx$ and let $C_{c}(G)$ be the convolution algebra of continuous, complex-valued functions on $G$ with
compact support. Let $H$ be a compact subgroup of $G$ with normalized Haar measure
$dh$ and denote by $C_{c}^{\#}(G)$ (or $C_{c}(H\backslash G/H )$) the space of functions
in $C_{c}(G)$ which are bi-invariant with respect to $H$, i.e. functions f which satisfy
\[f(hxh') = f(x),\quad \text{for}\ x\in G; h,h' \in H.\]
The space $C_{c}^{\#}(G)$ is a subalgebra of the convolution algebra $C_{c}(G)$ and is not commutative, in general. 
\begin{definition}[Gelfand pair]
    The pair $(G,H)$ is said to be a \textbf{Gelfand pair} if the convolution
algebra $C_{c}(H\backslash G/H )$ is commutative.
\end{definition}
 An important example of a Gelfand pair: when $G$ is any LCA group and $H$ is just the identity subgroup, $\{e\}.$ We now define the notion of \textit{spherical functions}, which generalizes characters.
 \begin{definition}[Spherical functions]
     Let $(G,H)$ be a Gelfand pair and let $\varphi$ be a continuous, bi-$H$-invariant function on $G$. The function $\varphi$ is called a \textbf{spherical function} if the functional $\chi$ defined by 
     \[
     \chi(f) = \int_{G} f(x) \varphi(x^{-1}) dx
     \]
     defines a non-trivial character on the convolution algebra $C_{c}(H\backslash G/H )$, i.e., $\chi(f* g)= \chi(f)\chi(g)$ for all $f,g\in C_{c}(H\backslash G/H ).$
 \end{definition}
\begin{definition}[Class one representation]
    A unitary representation $(\pi,\mathcal{H})$ of $G$ is said to be of class one if the subspace of $H$-fixed vectors $\mathcal{H}_{e}$ is non-trivial. 
\end{definition}
Typically, the notion of a class one representation is assumed to be irreducible, e.g. as in the treatments of van Dijk \cite{vanDijk} and Dieudonn\'e \cite{dieudonne}, in the definition, but we will need the definition for general unitary representations. 
\begin{theorem}
    Let $(G,H)$ be a Gelfand pair. The positive-definite spherical
functions on $G$ correspond one-to-one to the equivalence classes of irreducible
unitary representations of $G$ of class one.
\end{theorem}

In Theorem \ref{cpdthm}, we give an alternative method of obtaining bi-invariance of Class I representations that is essentially a once line triviality in the language of $C^*$-algebras via a Tomiyama type theorem. In general, adopting the methods of operator algebras / quantum information theory gives a clean perspective on positive definite functions.
        \subsubsection{Relation to relative property $(T)$}
        We follow Brown--Ozawa \cite{brown-ozawa} in this section.
        \begin{definition}
            Let $G$ be a group and $(\pi, \mathcal{H})$ be a unitary representation of $G$. A vector $h\in \mathcal{H}$ is said to be $G-invariant$ if $\pi(g)h=h$ for all $g\in G$. A net $(h_n)$ of unit vectors is almost invariant if $\lim \|\pi(g)h_n-h_n\| = 0$ for every $g\in G$. For a subset $E\subset G$ and $\epsilon>0,$ we say a nonzero $h\in \mathcal{H}$ is \textbf{$(E, \epsilon)$-invariant} if 
            \[
            \sup_{g\in E} \|\pi(g)h-h\| < \epsilon \|h\|.
            \]
        \end{definition}
        In the language of Fell topology, a unitary representation $\pi$ has a nonzero invariant vector if and only if the trivial representation is contained in $\pi$, and $\pi$ has almost invariant vectors if and only if the trivial representation is weakly contained in $\pi$.
        \begin{definition}
            Let $\Gamma \subseteq G$ be a subgroup. We say that the inclusion $\Gamma \subseteq G$ has \textbf{relative property (T)} if any unitary representation $(\pi, \mathcal{H})$ of $G$ which has almost $G$-invariant vectors has a nonzero $\Gamma$-invariant vector. We say $G$ has \textbf  {Kazhdan's property (T)} if the identity inclusion of $G$ into itself has relative property (T).
        \end{definition}
        Note that, (or see \cite{Bekka-Harpe-Valette}[Remark 1.4.4]) if a closed subgroup $H$ of a topological group $G$ has property (T), then the pair $(G,H)$ has the relative property (T).

            A well-known fact is that every compact (Lie) group $G$ has property (T). One way to see it is as follows: Suppose that we have an almost invariant (unit) vector $v$ for a unitary representation $(\pi, \mathcal{H})$. That is, $\sup_{g\in G} \|\pi(g)v - v\| < \epsilon,$ for some $\epsilon \in (0,1).$ Then consider the Bochner integral
            \[
            v_{0} = \int_{G} \pi(g) v\ d\mu(g),
            \]
            where $\mu$ is the (left)-invariant Haar measure on $G$. 
            Note that, 
            \begin{align*}
                \|v-v_{0}\| &= \|v - \int_{G} \pi(g)v\ d\mu(g)\|\\
                &\leq \int_{G} \|v- \pi(g)v\| \ d\mu(g) < \epsilon.
            \end{align*}
        Therefore, $\|v_0\| > 1-\epsilon >0$ i.e., $v_0$ is nonzero.
            Now for any $h\in G$,
            \begin{align*}
                \pi(h) v_{0} &= \pi(h) \int_{G} \pi(g) v\ d\mu(g) \\
                &=  \int_{G} \pi(hg) v\ d\mu(g) =  \int_{G} \pi(g) v\ d\mu(h^{-1} g)\\
                &= \pi(h) \int_{G} \pi(g) v\ d\mu(g) = v_{0}.
            \end{align*}
            Thus, $G$ is a Kazhdan group or has Kazhdan's property (T). Combining this fact with the above remark, we see that $(SO(n), SO(n-1))$  has relative property (T).

            We shall need the following characterization of relative property (T)
            \cite[Theorem 12.1.7 $(1)\Leftrightarrow (3)$]{brown-ozawa}
            \begin{theorem}[Brown--Ozawa]\label{unifT}
                Let $\Gamma \subseteq G$ be a subgroup.
                The following are equivalent:
                \begin{enumerate}
                    \item The inclusion $\Gamma \subseteq G$ has relative property (T),
                    \item Any sequence of positive definite functions on $G$ converging to the 
                    constant function $1$ is uniform on $\Gamma.$
                \end{enumerate}
            \end{theorem}
        \subsection{Existence of noncontinuous representations}
            We caution that, in general, a connected group $G$  will have many discontinuous Class I representions.
            For example suppose there is a nonconstant real-valued positive definite function such that $f(1)=1$, then taking 
            the limit of powers of $f$ gives a function which is $1, -1, 0$ valued.
    \subsection{$SO(\infty)$}
        Let $\hilbert$ be a infinite dimensional separable Hilbert space.
        Fix an orthonormal basis $e_1, e_2, \ldots.$
        There is a natural inclusion of $SO(n)$ into $\BH$
        by letting $SO(n)$ act naturally on the first $n$ basis vectors, and act trivially
        on the remaining. Moreover, these actions are mutually compatiable with the natural inclusion from $SO(n-1)$ to $SO(n)$ given by 
            \[\gamma \mapsto \begin{bmatrix}
                \gamma & 0 \\ 0 & 1
            \end{bmatrix}.\]
        Taking $\bigcup SO(n)$ viewed naturally as a subset of $\BH$ gives a group
        which we denote by $SO(\infty).$
        Naturally the group $SO(\infty)$ has the structure of a direct limit, and, as is typical
        of direct limits, it is somewhat topologically pathological. Thus, to give $SO(\infty)$
        the structure of a locally compact group, (which, for example, ensures the existence of faithful, perhaps infinite dimensional, representations) we endow it with the discrete topology.
        The group $SO(\infty)$ has a natural subgroup
        $SO(\infty)_{+}$ given by elements 
        \[\gamma \mapsto \begin{bmatrix}
                1 & 0 \\ 0 & \gamma
        \end{bmatrix}.\]
        The space of cosets $SO(\infty) / SO(\infty)_+ = S^{\infty}$ is called the {\bf ind-finite sphere}.
        Each double coset $SO(\infty)_+ \gamma SO(\infty)_+$
        is determined by the $11$ entry of $\gamma,$ which we denote by $\gamma_{11}.$

        \section{Discrete Schoenberg's theorem} \label{discoscho}

        \subsection{The Hamming Cube}

Let $X = \{-1,1\}^n$ be the Hamming cube. For $x, y \in X$, define the normalized inner product:
\[
\langle x, y \rangle = \frac{1}{n} \sum_{k=1}^n x_k y_k.
\]
The values of this inner product are discrete:
\[
\langle x, y \rangle \in \left\{1, 1 - \frac{2}{n}, 1 - \frac{4}{n}, \dots, -1 \right\}.
\]
The Hamming distance $d(x,y) = \#\{k : x_k \ne y_k\}$ relates to the inner product by:
\[
\langle x, y \rangle = 1 - \frac{2}{n} d(x,y).
\]

\subsection{Positive Definite Functions on the Hamming Cube}

\begin{definition}
A function $f: [-1,1] \to \mathbb{R}$ is \textbf{positive definite} on the Hamming cube if for every finite set $x_1, \dots, x_m \in X$, the matrix
\[
(f(\langle x_i, x_j \rangle))_{i,j=1}^m
\]
is positive semidefinite.
\end{definition}

\subsection{Zonal Spherical Functions and Krawtchouk Polynomials}

The Hamming cube is a 2-point homogeneous space. The zonal spherical functions are the \textbf{Krawtchouk polynomials} $K_j^{(n)}(d)$, defined for $j = 0, 1, \dots, n$ by:
\[
K_j^{(n)}(d) = \sum_{k=0}^j (-1)^k \binom{d}{k} \binom{n-d}{j-k}.
\]
The discrete analogue of Schoenberg's theorem is:

\begin{theorem}
A function $f: [-1,1] \to \mathbb{R}$ is positive definite on the Hamming cube $\{-1,1\}^n$ if and only if
\[
f\left(1 - \frac{2d}{n}\right) = \sum_{j=0}^n a_j K_j^{(n)}(d), \quad \text{with } a_j \geq 0.
\]
\end{theorem}

\subsection{The Limit $n \to \infty$}

The infinite analog of the Hamming cube requires slight introduction.

One would hope that the 
natural space is either the set of all finite bitstrings or perhaps the list of all infinite bitstrings.
However, if any sequence in such a space with infinitely many zeros or ones is given then, one can permute finitely many entries to get a sequence with zeros in any finite set of entries you want. Thus, there is a sequence going
to zero such that any positive definite function at the original sequence takes the same value, and thus the value at the zero
sequence.

Take $(X,\mathcal{A},\mu)$ a probability space and $\Sigma$ be a group of measurable bijections $f:X\rightarrow X$
such that
\begin{itemize}
    \item $\mu$ is non-atomic.
    \item If $A_1, \ldots, A_n$  and $B_1, \ldots, B_n$ are measurable partitions of $X,$ such that
    $\mu(A_i)=\mu(B_i),$ then there is $\sigma \in \Sigma$ such that $\sigma(A_i)=B_i,$
\end{itemize}
is called an {\bf infinite Hamming cube.}
Here the analogue of bitstrings are the measurable sets $\mathcal{A},$ which forms a group under symmetric complement.
Thus, our underlying double coset space is the quotient of $\mathcal{A}$ by the action of $\Sigma$ are exactly
the real numbers in $[0,1].$
Thus, the inner product given by
$$\langle E,F \rangle= 2 - 2\mu(E\Delta F),$$
which essentially agrees with the Hamming inner product for sets rational measure with some fixed denominator.
That is, for each partition $A_1, \ldots, A_n$ such that $\mu(A_i) =1/n$ there is a natural map from $\{0,1\}^n.$

We examine the asymptotic behavior of the Krawtchouk polynomials as $n \to \infty$. We utilize the normalized coordinate $u \in [-1,1]$ defined by $d = \frac{n(1-u)}{2}$.

\begin{lemma}\label{Krawtchouk limit}
Fix an integer $j \geq 0$. For any $u \in [-1, 1]$, let $d_n$ be a sequence such that $1 - 2d_n/n \to u$. Then,
\[
\lim_{n \to \infty} \frac{j!}{n^j} K_j^{(n)}(d_n) = u^j.
\]
\end{lemma}

\begin{proof}
We start by getting the asymptotic bounds on the Krawtchouk polynomials using the following explicit formula:
\[
K_j^{(n)}(d) = \sum_{k=0}^j (-1)^k \binom{d}{k} \binom{n-d}{j-k}.
\]
For any variable $X$ and fixed integer $m$, the Stirling approximation gives, $\binom{X}{m} = \frac{X^m}{m!} + O(X^{m-1})$.
Substituting $d = \frac{n(1-u)}{2}$ and $n-d = \frac{n(1+u)}{2}$,
\begin{align*}
\binom{d}{k} &= \frac{1}{k!} \left( \frac{n(1-u)}{2} \right)^k + O(n^{k-1}), \\
\binom{n-d}{j-k} &= \frac{1}{(j-k)!} \left( \frac{n(1+u)}{2} \right)^{j-k} + O(n^{j-k-1}).
\end{align*}
Multiplying these terms, the leading order term is of order $n^k \cdot n^{j-k} = n^j$. The cross terms involving error terms are of order $O(n^{j-1})$. Thus
\[
\binom{d}{k} \binom{n-d}{j-k} = \frac{n^j}{k!(j-k)!} \left( \frac{1-u}{2} \right)^k \left( \frac{1+u}{2} \right)^{j-k} + O(n^{j-1}).
\]
Substituting back into the sum gives us
\[
K_j^{(n)}(d) = \sum_{k=0}^j (-1)^k \left[ \frac{n^j}{k!(j-k)!} \left( \frac{1-u}{2} \right)^k \left( \frac{1+u}{2} \right)^{j-k} + O(n^{j-1}) \right].
\]
Isolating the $n^j$ term and rearranging the constants:
\[
K_j^{(n)}(d) = n^j \left( \sum_{k=0}^j \frac{(-1)^k}{k!(j-k)!} \left( \frac{1-u}{2} \right)^k \left( \frac{1+u}{2} \right)^{j-k} \right) + O(n^{j-1}).
\]
Multiplying by $j!/n^j$, the error term becomes $O(n^{-1})$, which vanishes as $n \to \infty$. Therefore
\[
\lim_{n \to \infty} \frac{j!}{n^j} K_j^{(n)}(d) = \sum_{k=0}^j \frac{j! (-1)^k}{k!(j-k)!} \left( \frac{1-u}{2} \right)^k \left( \frac{1+u}{2} \right)^{j-k}.
\]
Now we apply the Binomial Theorem to $(A+B)^j$ where $A = -\frac{1-u}{2}$ and $B = \frac{1+u}{2}$:
\[
\text{Sum} = \left( - \frac{1-u}{2} + \frac{1+u}{2} \right)^j = \left( \frac{u-1+1+u}{2} \right)^j = u^j.
\]
\end{proof}

Taking a limit of partitions of size $n!$ (to obtain common divisors) we see that on measurable sets with rational measure this determines the values everywhere.

\begin{corollary}
As $n \to \infty$, a function $f$ is positive definite on an infinite Hamming cube if and only if it admits a power series expansion with nonnegative coefficients:
\[
f(u) = \sum_{j=0}^\infty a_j u^j, \quad a_j \geq 0.
\]
If $f(1) < \infty$, this series converges uniformly on $[-1,1]$.
\end{corollary}

\begin{proof}
Let $f$ be the limit of a sequence of positive definite functions. Then $f(u) = \sum a_j u^j$ with $a_j \ge 0$.
The value at $u=1$ is $f(1) = \sum a_j$. Since $f(1) < \infty$, the series of coefficients converges absolutely.
By Abel's Theorem (extended to closed intervals for non-negative coefficients), the power series $\sum a_j u^j$ converges uniformly on $[-1,1]$ because $\sum a_j \sup_{u \in [-1,1]} |u^j| = \sum a_j < \infty$.
\end{proof}

Unless the cardinality of the continuum is a real-valued measurable cardinal, $\mathcal{A}$ is unlikely to be all subsets of your space \cite{ulam1930masstheorie}. We will see later via Arveson's extension theorem that there are extensions of our natural inner product to the whole space, although there is no general agreement with the measure structure.

\section{Completely positive maps} \label{completepositivity}
    Recall a Banach algebra is a a normed algebra which is complete.
    A $C^*$-algebra is a Banach algebra that has a conjugate linear involution $X\mapsto X^*$
    such that $\|X\|^2 = \|X^*X\|.$
    Let $\mathcal{H}$ be a Hilbert space.
    We say $\pi: \mathcal{A}\rightarrow B(\mathcal{H})$ is a representation 
    if it is an algebra homomorphism such that $\pi(a)^*=\pi(a^*).$
    The Gelfand-Naimark-Segal construction establishes that 
    any $C^*$-algebra has an isometric injective representation, and thus we may view $C^*$-algebras
    naturally as closed subalgebras of $B(\mathcal{H}).$
    
    An element $H$ is called positive semidefinite if there is $X$ such that $H=X^*X.$

    A linear map $\varphi: \mathcal{A}\rightarrow \mathcal{B}$
    is said to be completely positive if
    $[\varphi(X_{ij})]_{i,j}$ is positive semidefinite whenever $[X_{ij}]_{ij}$ is positive semidefinite.

    The set of states on a $C^*$-algebra are the set of completely positive maps $\varphi: \mathcal{A}\rightarrow \mathbb{C}$
    such that $\varphi(1)=1.$ These form a compact convex set, whose extreme points are called pure states.
    
    We now recall Dixmier's version of Glimm's lemma.
    \begin{theorem}[\cite{dixmier}] \label{glimm}
        Let $\mathcal{A}$ be a $C^*$-algebra.
        Let $\Pi$ be a family of representations of
        $\mathcal{A}$ so that the intersection of their kernels is $0.$
        Any pure state on $\mathcal{A}$ is a weak$^*$ limit of states of the form
        \[
            \varphi(g)= \langle \pi(g) \xi,\xi\rangle
        \]
        where $\pi\in \Pi.$
    \end{theorem}
    \subsection{The Stinespring representation theorem}
        The following give the representer theroem for completely positive maps.
        \begin{theorem}[Stinespring representation theorem]
            Let $\mathcal{M}$ be a C$^*$-algebra.
            A map $\varphi: \mathcal{M} \rightarrow \BH$ is completely positive if and only if
            $\varphi(x) = V^*\pi(x)V$ where $\pi: \mathcal{M} \rightarrow \BK$ is a representation.
        \end{theorem}
    \subsection{The Arveson extension theorem}
        An operator system is a subset of a $C^*$-algebra that is $*$-closed and contains the identity. A completely positive map is analogous to the definition for $C^*$-algebra except constraining your choice of block entries to the operator system. There is a natural way to lift completely positive maps on operator systems to the entire $C^*$-algebra.
        \begin{theorem}[Arveson extension theorem]
            Let $\mathcal{M}$ be a C$^*$-algebra.
            Let $X\subseteq \mathcal{M}$ be an operator system. Let
            $\varphi: X \rightarrow \BH$ be completely positive.
            Then there exists $\tilde{\varphi}: \mathcal{M} \rightarrow \BH$
            such that $\tilde{\varphi}|_X=\varphi.$
        \end{theorem}

        We refer the interested reader to Paulsen \cite{paulsen2002completely}.
    \subsection{Homomorphic Tomiyama theorem}
        We need the following homomorphic version of Tomiyama's theorem \cite{tomiyama1957projection}, (also analyzed in \cite{pluta2017homomorphic})
        as given in \cite{pascoe2019cauchy} which gives some automatic bimodularity for completely positive maps.
        \begin{theorem}[Homomorphic Tomiyama theorem]
            Let $\mathcal{A}$ be a $C^*$-algebra, Let $B\subseteq A$
            be a subalgebra. Suppose $E: \mathcal{A} \rightarrow \mathcal{M}$ is completely positive and $E|_{\mathcal{B}}$
            is a unital homomorphism.
            Then,
                $$E(b_1ab_2)=E(b_1)E(a)E(b_2)$$
            for all $a \in \mathcal{A}$ and $b_1,b_2 \in \mathcal{B}.$
        \end{theorem}
    \subsection{Completely positive maps on group algebras}
        \begin{definition}
            Let $G$ be a locally compact group.
            We define the group algebra $\mathbb{C}[G]$ to be the set of formal sums
            \[
                \sum a_ig_i
            \]
            where $g_i\in G$ and $a_i \in \mathbb{C}.$
            There is a natural involution on $\mathbb{C}[G]$ which satisfies $g^* = g^{-1}.$
            There is a natural representation of $\mathbb{C}[G]$ given by 
                \[\|\sum a_ig_i\| = \sup_{\Pi(G)} \|\sum a_i\pi(g)\|\]
            where $\Pi(G)$ is the set of all continuous unitary representations of $\mathbb{G}.$
            We define the complete group algebra $\overline{\mathbb{C}[G]}$ to be the completion of $\mathbb{C}[G]$ with respect to $\|\cdot\|.$
        \end{definition}
        We note that in general the complete group algebra $\overline{\mathbb{C}[G]}$ is not the same object as the full group $C^*$-algebra $C^*(G)$ of a locally compact group, which is defined in terms of continuous functions on the group, unless, of course, the group is discrete. The necessity of dealing with the group algebra itself arises directly from the interpolation problems we will consider later-- we want to understand functions which ar positive definite on some (often finite) subset of a homogeneous space $G/H$ and so our interpolation data may not be enough to induce a map on an algebra of continuous functions. For the most part we merely need to work in $\mathbb{C}[G]$ itself and not its completion.

        The following Theorem of Helton and McCullough will be useful to us.
        Such noncommutative sums of squares were themselves inspired by the work Schm\"udgen \cite{schmud,schmudsurv} and Putinar \cite{putinarposss} on positive polynomials in several variables, which itself was inspired by the classical work of Fej\'er \cite{fejer} and Riesz \cite{riesz1916}.
        \begin{theorem}[Helton, McCullough \cite{heltmc}]
            Let $\hilbert$ be an infinite dimensional separable Hilbert space.
            Let $\mathbb{C}<x_1,\ldots, x_d, x_1^*,\ldots,x_d^*>$ denote the noncommutative
            polynomials in the variables $x_1,\ldots,x_d$
            and their adjoints $x_1^*, \ldots, x_d^*.$
            Let $\mathcal{P}$ be a subset of
            selfadjoint elements of
            $\mathbb{C}<x_1,\ldots, x_d, x_1^*,\ldots,x_d^*>$
            containing elements of the form $C_i-x_i^*x_i,$ where $C_i\in \mathbb{R}.$
            Let 
                \[\mathcal{D}_\mathcal{P}= \{X \in \BH^d | p(X) \textrm{ is positive semidefinite for all }p\in \mathcal{P}\}\]
            An element $p\in M_n(\mathbb{C}<x_1,\ldots, x_d, x_1^*,\ldots,x_d^*>)$ is positive definite on $\mathcal{D}_{\mathcal{P}}$ if and only if 
            \[p = \varepsilon + \sum^N_{i=1} a_i^*a_i + \sum^M_{j=1} b_j^*p_jb_j\]
            for some $\varepsilon>0, a_i,b_j \in M_{1n}(\mathbb{C}<x_1,\ldots, x_d, x_1^*,\ldots,x_d^*>)$
            and $p_i\in \mathcal{P}.$
        \end{theorem}

        The Helton-McCullough Positivstellensatz formally-speaking required finitely many variables.
        However, an appropriate reformulation holds for arbitrary collections of indeterminants.
        \begin{theorem}[Large Helton-McCullough Positivstellensatz]
            Let $\mathbb{X}$ be a collection of formal letters of any cardinality.
            Let $\hilbert$ be a Hilbert space with Hilbert space dimension at least $|\mathbb{X}\times \mathbb{N}|.$
            Let $\mathbb{C}<\mathbb{X}, \mathbb{X}^*>$ denote the noncommutative
            polynomials in the letters of $\mathbb{X}$
            and their adjoints $\mathbb{X}^*.$
            Let $\mathcal{P}$ be a subset of
            selfadjoint elements of
            $\mathbb{C}<\mathbb{X}, \mathbb{X}^*>$
            containing elements of the form $C_x-x^*x,$ where $C_x\in \mathbb{R}$
            for each $x\in \mathbb{X}.$
            Let 
                \[\mathcal{D}_\mathcal{P}= \{X \in \BH^d | p(X) \textrm{ is positive semidefinite for all }p\in \mathcal{P}\}\]
            An element $p\in M_n(\mathbb{C}<\mathbb{X}, \mathbb{X}^*>)$ is positive definite on $\mathcal{D}_{\mathcal{P}}$ if and only if 
            \[p = \varepsilon + \sum^N_{i=1} a_i^*a_i + \sum^M_{j=1} b_j^*p_jb_j\]
            for some $\varepsilon>0, a_i,b_j \in M_{1n}(\mathbb{C}<\mathbb{X}, \mathbb{X}^*>)$
            and $p_i\in \mathcal{P}.$
        \end{theorem}
        \begin{proof}
            Suppose $p$ is not of the above form.
            Order the finite subsets $\Lambda$ of $\mathbb{X}$ by inclusion.
            For each $\Lambda,$ let $\mathcal{P}_{\Lambda}$ denote the subcollection of $\mathcal{P}$
            only using letters from $\Lambda.$
            As $p$ is not of the above form, a fortiori it is not of such a form for the corresponding subcollection of variables containing the variables of $\Lambda,$
            so by the Helton-McCullough Positivstellensatz, 
            there exists $X_{\Lambda} \in \mathcal{D}_{\mathcal{P}_{\Lambda}}$
            such that $p(X_{\Lambda})$ is not positive definite. Namely, there are orthonormal vectors
            $v_1, \ldots, v_n \in \hilbert$ such that $[\langle p_{ij}(X_\Lambda)v_i,v_j\rangle ]_{i,j}$ is not positive definite.
            Thus, the map $\sigma_{\Lambda}$ from $\mathbb{C}<\Lambda, \Lambda^* >$
            given by $q \mapsto [\langle q(X_\Lambda)v_i,v_j \rangle ]_{i,j}$ satisfies
            that $\sigma_{\Lambda}$ is positive semidefinite on $\mathcal{P}_{\Lambda},$
            and  that $[\sigma(q_i^*q_j)]_{i,j}$ is positive semidefinite for any choice of $q_i$ and $\sigma_{\Lambda}(1)$ is the identity.
            Let $\sigma$ be a limit point with respect to the net $\sigma_{\Lambda}.$
            Note $\sigma$ is unital and $\sigma_{\Lambda}$ is positive semidefinite on $\mathcal{P}_{\Lambda}.$
            Note $[\sigma(p_{ij})]_{i,j}$ is not positive semidefinite.
            Thus, applying the Gelfand-Naimark-Segal construction to $\sigma$
            gives an $X \in \mathcal{D}_{\mathcal{P}}$ such that $p(X)$ is not positive definite.
        \end{proof}

        Taking $\mathcal{P}$ to be the collection of self-adjoint polynomials
        containing $\pm(1-x_g^*x_g), \pm(1-x_gx_g^*), -(x_g^*-x_{g^{-1}})^*(x_g^*-x_{g^{-1}})$ over all $g\in G,$ $-(x_gx_{\tilde{g}}-x_{g\tilde{g}})^*(x_gx_{\tilde{g}}-x_{g\tilde{g}})$ over all pairs, we immediately obtain the following.
        \begin{lemma} \label{positivityofgrpelt}
            Let $G$ be a discrete group.
            If $p\in M_n(\mathbb{C}[G])$ is positive definite, then
            $p = \sum^N_{i=1} a_i^*a_i$
            for some $a_i \in M_{1n}(\mathbb{C}[G]).$
        \end{lemma}
        Other results for sums of squares in group algebras have been obtained by Netzer and Thom
        \cite{netzer2013real}.
        We note that Mehta, Slofstra and Zhao \cite{mehta2023positivity} have considered undecidability questions with respect to determining sums-of-squares in group algebras in detail, which is somewhat inherited from undecidability of the word problem.
        That is, the Helton-McCullough result is inherently nonconstructive. 

        The following gives a condition for a map on a group algebra coming from a discrete group to be completely positive.
        \begin{lemma}\label{lemmaposcondinit}
        Let $G$ be a discrete group.
        A map $\varphi: \mathbb{C}[G] \rightarrow \mathbb{C}$
        is completely positive if and only if given any $g_1, \ldots, g_N,$
        the matrix $[\varphi(g_i^{-1}g_j)]_{i,j}$ is positive semidefinite.
        \end{lemma}
        \begin{proof}
            Let $M=[m_{ij}]_{i,j}$ be a positive definite matrix over $\mathbb{C}[G]$
            As $M$ is positive definite, $M = \sum^T_{t=1} a_t^*a_t$ as in
            Lemma \ref{positivityofgrpelt}
            where $A$ is some rectangular matrix with entries in $\mathbb{C}[G].$ 
            Note we may simulatenously write \[a_t = \begin{bmatrix}
                g_1 & \ldots & g_N
            \end{bmatrix} A_t\]
            where $A_i$ is a matrix with complex number entries.
            Thus 
                $M = \sum A_t^*[g_i^{-1}g_j]A_t$
            and so $\varphi(M)= \sum A_t^*[\varphi(g_i^{-1}g_j)]A_t$
            which is positive by assumption.
        \end{proof}

        We note that Ozawa \cite{ozawa2016noncommutative}, followed by Netzer and Thom \cite{netzer2015kazhdan} have used similar semidefinite programming type ideas to characterize property (T).


\section{Invariant completely positive definite functions on homogeneous spaces} \label{homogeneous}
    We now define the notion of a completely positive definite function, which generalize classic positive definite functions, but also allow for a highly austere treatment along the lines of Bochner using complete positivity.
    \begin{definition}
        Let $\hilbert$ be Hilbert space.
        Let $G$ be a locally compact group and let $H$ be a closed subgroup of $G.$
        We say a function $\varphi: G/H \rightarrow \BH$
        is {\bf completely positive definite} if
        the induced linear map satisfying $g \mapsto \varphi(g / H)$
        on the group algebra $\mathbb{C}[G]$ is completely positive.
    \end{definition}
    We note that the matrix positive definite functions play a role in packings with several radii \cite{de2014upper}. In fact, to prove the Cohn--Elkies and Kabatiansky--Levenshtein counterparts of the several radii version needs the Schoenberg-type result meant for matrix pd functions.

    We first note that any completely positive definite function on $G/H$ is in fact $H$-invariant, and thus the induced completely positive map is bi-invariant.
    \begin{proposition}
        Let $\hilbert$ be Hilbert space.
        Let $G$ be a group and let $H$ be a subgroup of $G.$
        Let $\varphi: G/H \rightarrow \BH$ be completely positive definite. 
        Then, \[\varphi(h \cdot x) = \varphi(x)\]
        for all $h\in H$ and $x \in G/H.$
    \end{proposition}
    \begin{proof}
        Without loss of generality, assume $\varphi(e/H)=1.$ Note that the induced completely positive map $\hat{\varphi}$ is constant on $H$ as
        $h \mapsto \varphi(h/H) = \varphi(e/H) = 1.$
        Let $h \in H, x\in G/H$ and $\hat{x}\in G$ such that
        $\hat{x} /H = x.$
        Thus, by the homomorphic Tomiyama theorem we have that \[\hat{\varphi}(h)\hat{\varphi}(\hat{x})=\hat{\varphi}(h\cdot x)\]
        and so
        \[\varphi(h\cdot x) = \hat{\varphi}(h\hat{x})=\hat{\varphi}(h)\hat{\varphi}(\hat{x})=
        \varphi(h / H)\varphi(x)=\varphi(x).\]
    \end{proof}

    We now give a representation theorem for a general completely positive definite function on $G/H.$
    \begin{theorem}\label{cpdthm}
        Let $\hilbert$ be Hilbert space.
        Let $G$ be a group and let $H$ be a subgroup of $G.$
        Let $\varphi: G/H \rightarrow \BH.$
        The map $\varphi$ is completely positive definite if and only if there exists a representation of class one $\pi$ of $G$ and an isometry $V$ such that the range of $V$ is a subspace of the $H$-invariant vectors of $\pi$ and 
            \[\varphi(g/H) = V^*\pi(g)V.\]
    \end{theorem}
    \begin{proof}
        Let $\hat{\varphi}$ be the induced completely positive map. By the Stinespring representation theorem, there is a minimal representation $\pi$ of $\mathbb{C}[G]$ such that 
        $\hat{\varphi}(g) = V^*\pi(g)V.$
        Note that $\pi(h)$ restricted to the range of $V$
        is the identity when $h \in H$ as by minimality
        vectors of the form $\pi(g)Vw$ span the space
        and, for any $u, w \in \hilbert,$ we have that
        \begin{align*}
            \langle \pi(h)Vu, \pi(g)Vw \rangle
            &= \langle V^*\pi(g)^*\pi(h)Vu, w\rangle\\
            &= \langle V^*\pi(g^{-1}h)Vu, w\rangle\\
            &= \langle \varphi(g^{-1}/H)u, w\rangle\\
            &= \langle V^*\pi(g^{-1})Vu, w\rangle\\
            &= \langle Vu, \pi(g)Vw \rangle,
        \end{align*}
        and, thus, $\pi(h)Vu = Vu$- that is, as 
        vectors of the form $Vu$ exhaust the range of $V,$
        the range of $V$ is invariant under each $\pi(h).$
    \end{proof}
    \subsection{The ind-finite sphere and consequences for representations of Class I of $SO(\infty)$}
        Christensen and Ressel established the following theorem following the work of Schoenberg.
        \begin{theorem}
            Let $\hilbert$ be an infinite dimensional Hilbert space, let $S$ be its unit sphere.
            Let $f:[-1,1]\rightarrow \mathbb{R}$ such that 
            \begin{enumerate}
                \item $f(1)=1,$
                \item $f(\langle x, y \rangle)$ defines a positive kernel $S \times S,$
            \end{enumerate}
            then $f$ is a positivity preserver and is of the form
                \[f(x) = a\chi(x=1 \vee x=-1) + bx\chi(x=1 \vee x=-1) +  \sum c_ix^i\]
            where $a,b,c_i \geq 0, a+b \sum c_i =1$ and $\chi$ is an indicator function.
        \end{theorem}

        Strictly speaking, we need their result for the ind-finite sphere $S^\infty$, which is slightly smaller than $S.$ However, the property that $f(\langle x, y \rangle)$ defines a positive kernel requires only that for each finite list of $x_1,\ldots, x_n \in S$
        that the matrix $[f(\langle x_i, x_j \rangle)]_{i,j}$ be positive semidefinite, and since we may choose a unitary $U$ such that $Ux_i$ is finitely supported and
        $[f(\langle x_i, x_j \rangle)]_{i,j}=[f(\langle Ux_i, Ux_j \rangle)]_{i,j},$ we see the following corollary
        \begin{corollary} \label{indfinitepd}
            Let $S^\infty$ be the ind-finite sphere.
            Let $f:[-1,1]\rightarrow\mathbb{R}$ such that 
            \begin{enumerate}
                \item $f(1)=1,$
                \item $f(\langle x, y \rangle)$ defines a positive kernel $S^\infty \times S^\infty,$
            \end{enumerate}
            then $f$ is a positivity preserver and is of the form
                \[f(x) = a\chi(x=1 \vee x=-1) + bx\chi(x=1 \vee x=-1) +  \sum c_ix^i\]
            where $a,b,c_i \geq 0, a+b \sum c_i =1$ and $\chi$ is an indicator function.
        \end{corollary}

        Now, by Theorem \ref{cpdthm}, we see that any completely positive definite function on
        $\mathbb{C}[SO(\infty)]$ is of the form $V^*\pi(g)V,$ and hence any extreme state state corresponding to an $SO(\infty)_+$-invariant vector must be of the form $\gamma_{11}^n$
        or $\chi(\gamma_{11}=1 \vee \gamma_{11}=-1)$ or $\gamma_{11}\chi(\gamma_{11}=1 \vee \gamma_{11}=-1).$
        \begin{corollary}\label{compresthm}
            Any irreducible class one representation of $SO(\infty)$
            arises from one of the following states:
            \begin{enumerate}
                \item $\varphi(\gamma)=\gamma_{11}^n$
                \item $\varphi(\gamma)=\chi(\gamma_{11}=1 \vee \gamma_{11}=-1)$ 
                \item $\varphi(\gamma)=\gamma_{11}\chi(\gamma_{11}=1 \vee \gamma_{11}=-1)$ 
            \end{enumerate}
        \end{corollary}
        It is somewhat remarkable as we have endowed $SO(\infty)$ with the discrete topology.
        \subsection{Locally compact groups modulo subgroups with property (T)}
            We now endeavor to show that for certain kinds of homogeneous spaces, that one may approximate positive definite functions as pointwise limits of continuous functions. First, we need the following representation theoretic lemma.
            \begin{lemma}
                Let $G$ be a locally compact group and let $H$ be a countable discrete subgroup with property (T). Then, any state arising from a class I representation of $G$ with respect to $H$ is a limit of states arising from continuous class I representations of $G$ with respect to $H.$
            \end{lemma}
            \begin{proof}
                As any compact subset of $H$ is finite, Dixmier's version of Glimm's lemma, given above as Theorem \ref{glimm} gives that any pure state on $G$ is a limit of states arising from continuous representations of $G.$
                Thus, for each finite subset $\Lambda$ of $G$ pick a sequence of states
                $\sigma_{\Lambda,n}(g) = v_{\Lambda,n}^*\pi_{\Lambda,n}(g)v_{\Lambda,n}$ which converge pointwise on $H\cup \Lambda.$
                As $H$ has property (T) with respect to itself such a sequence converges uniformly to $1$ on $H$ by Theorem \ref{unifT}. Hence, letting $n$ be smallest index such that $|\sigma - \sigma_{\Lambda,n}||_{H\cup \Lambda} \leq \frac{1}{|\Lambda|}$ and take
                $\tilde{v}_{\Lambda}$ to be the unique smallest norm element 
                of the closed convex hull $\mathfrak{C}_H$ of the elements $H.$ Such $\tilde{v}_{\Lambda}$ must have norm
                greater than $\sqrt{1-1/|\Lambda|}.$
                Note $h\tilde{v}_{\Lambda} \in \mathfrak{C}_H,$ and 
                $\|h\tilde{v}_{\Lambda}\|=\|\tilde{v}_{\Lambda}\|$ and thus
                $h\tilde{v}_{\Lambda} = \tilde{v}_{\Lambda}.$ So $\tilde{v}_\Lambda$
                is H-invariant and thus so is $v_\Lambda=\tilde{v}_\Lambda/\|\tilde{v}_{\Lambda}\|.$ 
                Letting $\sigma_{\Lambda} = v_{\Lambda}^*\pi_{\Lambda,n}(g)v_{\Lambda}$ and
                taking a weak$^*$ limit point with respect to finite subsets $\Lambda$ ordered by inclusion, we are done.
            \end{proof}

            Equipped with the above lemma, we see that when $G$ be a locally compact group and let $H$ be a countable discrete subgroup with property (T), pointwise approximation by continuous positive definite functions is possible.
            \begin{theorem}
                Let $G$ be a locally compact group and let $H$ be a countable discrete subgroup with property (T). Then, any positive definite function $f:G/H \rightarrow \mathbb{R}$
                is a pointwise limit of continuous positive definite functions.
            \end{theorem}

\section{Invariant partially defined completely positive definite functions on homogeneous spaces} \label{partialhomogeneous}
    \begin{definition}
        Let $G$ be a group and let $H$ be a subgroup of $G.$ Let $\Gamma \subseteq G/H.$ We say $\Gamma$ is {\bf symmetric} if whenever $g/H \in \Gamma$, then $g^{-1}/H \in \Gamma.$ We call $\Gamma$ {\bf Archimedian}
        if $e/H \in \Gamma.$

        We define $\mathbb{C}[\Gamma]\subseteq \mathbb{C}[G]$ to be the span of $g\in G$ such that $g/H \in \Gamma.$
    \end{definition}

    We now note that  $\mathbb{C}[\Gamma]$ is an operator system.
    \begin{proposition}
        Let $G$ be a group and let $H$ be a subgroup of $G.$ Let $\Gamma \subseteq G/H.$
        If $\Gamma$ is symmetric and Archimedian, then
        $\mathbb{C}[\Gamma]$ is an operator system.
    \end{proposition}
    \begin{proof}
        Note that the Archimedian property ensures $1 \in \mathbb{C}[\Gamma].$ Note that by the symmetric property, we have that if $g\in \mathbb{C}[\Gamma]$
        then $g^* = g^{-1} \in  \mathbb{C}[\Gamma].$
    \end{proof}

    We now describe a condition for a map from $\mathbb{C}[\Gamma]$ to be completely positive, adapting the condition from Lemma \ref{lemmaposcondinit}.
    \begin{lemma} \label{lemposcond}
        Let $G$ be a discrete group.
        Let $\Gamma \subseteq G$ such that $\Gamma=\Gamma^*, 1\in \Gamma.$
        A map $\varphi: \mathbb{C}[\Gamma] \rightarrow \mathbb{C}$
        is completely positive if and only for any $g_1, \ldots, g_N,$
        the matrix $[\varphi(g_i^{-1}g_j)]_{i,j}$
        where the entries such that $g_i^{-1}g_j \notin \Gamma$ are left unspecified
        is a partially defined positive semidefinite matrix.
        \end{lemma}
        \begin{proof}
            Let $M=[m_{ij}]_{i,j}$ be a positive definite matrix with entries from $\mathbb{C}[\Gamma].$
            As $M$ is positive definite, $M = \sum^T_{t=1} a_t^*a_t$ as in
            Lemma \ref{positivityofgrpelt}
            where $A$ is some rectangular matrix with entries in $\mathbb{C}[G].$ 
            Note we may simulatenously write \[a_t = \begin{bmatrix}
                g_1 & \ldots & g_N
            \end{bmatrix} A_t\]
            where $A_i$ is a matrix with complex number entries.
            Thus 
                $M = \sum A_t^*[g_i^{-1}g_j]A_t$
            and so $\varphi(M)= \sum A_t^*[\varphi(g_i^{-1}g_j)]A_t$
            where the unspecified entries cancel out.
            Complete the matrix 
                $[\varphi(g_i^{-1}g_j)]_{i,j}$
            to a positive semidefinite matrix to witness the positivity of $\varphi(M).$
        \end{proof}

    We now give a definition of a partially defined completely positive definite function.
    \begin{definition}
        Let $\hilbert$ be Hilbert space.
        Let $G$ be a group and let $H$ be a subgroup of $G.$ Let $\Gamma \subseteq G/H$ be symmetric and Archimedian.
        We say a $\varphi: \Gamma \rightarrow \BH$
        is  a {\bf partially defined completely positive definite function} if
        the induced linear map satisfying $g \mapsto \varphi(g / H)$
        on the group algebra $\mathbb{C}[\Gamma]$ is completely positive.
    \end{definition}
    Coupling the Arveson extension theorem with our characterization of completely positive definite functions gives a characterization of partially defined completely positive definite functions.
    \begin{theorem}\label{liftpd}
        Let $\hilbert$ be Hilbert space.
        Let $G$ be a group and let $H$ be a subgroup of $G.$ Let $\Gamma$ be symmetric and Archimedian.
        Let $\varphi: \Gamma \rightarrow \BH.$
        The map $\varphi$ is a completely positive definite function if and only if there exists a representation of class one $\pi$ of $G$ and an isometry $V$ such that the range of $V$ is a subspace of the $H$-invariant vectors of $\pi$ and 
            \[\varphi(g/H) = V^*\pi(g)V.\]
    \end{theorem}
    \begin{proof}
        Let $\hat{\varphi}$ be the induced completely positive map. By the Arveson extension theorem,
        $\hat{\varphi}$ extends to all of $\mathbb{C}[G]$
        and thus by Theorem \ref{cpdthm} we are done.
    \end{proof}
    \subsection{Interpolation for $SO(n)$ on finite sets}
    We now turn our attention to the following problem.
    Let $X \subseteq [-1,1]$ be a finite set containing $1.$ Let $f: X \rightarrow \BH.$
    When does $\hat{f}(\gamma)=f(\gamma_{11})$ define a 
    partially defined completely positive definite function?

    \begin{lemma} \label{convcomb}
        Any $P^{(n)}_iP^{(n)}_j$ is in the cone generated $P^{(n)}_k$ for $k\leq i+j.$
    \end{lemma}
    \begin{proof}
        Each $P^{(n)}_i$ has $P^{(n)}_i(\gamma_{11})$ is positive definite on $SO(n)/SO(n-1).$ 
        Note
        $P^{(n)}_i(\gamma_{11})P^{(n)}_j(\gamma_{11})$
        is a positive definite function by the Schur product theorem. 
        Thus by Schoenberg's theorem it must be a positive linear combination of $P^{(n)}_k(\gamma_{11}).$
        Note that if 
        where $k> i+j,$ $P^{(n)}_iP^{(n)}_j \perp P^{(n)}_k,$
        so we are done.
    \end{proof}

    First we desire to prove the following theorem.
    \begin{lemma}\label{hulllemma}
        Let $X \subseteq [-1,1]$ be a finite set containing $1.$
        Fix $n \geq 3.$
        Let $\hat{P}_i = \frac{P^{(n)}_i}{P^{(n)}_i(1)}|_{X}$
        There exists an  $N>0$ depending on $X$ such that any
        $\hat{P}_i$
        is in the convex hull of $\hat{P}_0, \ldots, \hat{P}_N.$
    \end{lemma}
    \begin{proof}
        Without loss of generality, assume $X = -X.$ Namely,
        we are assuming $-1 \in X.$ 

        Because of $\hat{P}_i$ is an odd function for odd $i$ and an even function for even $i$
        it must any odd $\hat{P}_i$ must be in the convex hull of odd $\hat{P}_j$, and similarly for even ones. Also note that for any $i,j$
        $\hat{P}_i\hat{P}_j$ is a convex combination of some $\hat{P}_k$ where $k \leq i+j$ by Lemma \ref{convcomb}.
        Note that by Darboux's estimate given as Theorem \ref{darbouxest}, $\lim \hat{P}^(i)(x) = 0$ for $x \neq \pm -1.$
        Thus it is sufficient to show that for each $x\in X$ such that $x>0$ there exist
        functions $g_x$ which are each
        a convex combinations  of products of the form
        $\hat{P}_{i_1}\ldots \hat{P}_{i_n}$ where $\sum i_n$ is even such that 
        $g_x$ is negative at $x$ and $0$ at each $y \in X \cap [0,1)$ such that $y\neq x,$
        as then $g_{x}^2$ is even and positive at $x,$ and any even function small enough away from $\{-1,1\}$ will be in the hull of these, and similarly, for odd functions they will be in the hull of $\hat{P}_1g_x$ and $\hat{P}_1g_x^2.$

        Note that for even $i$, the family $P^{(n)}_i(\sqrt{t})$ is itself a family of orthogonal polynomials. Thus, the interlacing condition and density conditions given in Theorem \ref{density} and Theorem \ref{interlacing} implies that for each $x\in X\cap[0,1)$ there exists an even $i_x$ such that $\hat{P}_{i_x}(x)<0.$ (Note that the choice of $i_x$ can be taken to only depend on the distance of $x$ from $1.$)
        Now define
                \[g_x = h_{x}\prod_{y\neq x, y\in X \cap [0,1)} \left(\frac{|e_{x,y}(y)h_y(y)|+e_{x,y}h_y}{|e_{x,y}(y)h_y(y)|+1}\right)^{2}\]
        where $e_{x,y} = 1$ if $h_y(x)\neq h_y(y)$ and $e_{x,y} = \hat{P}_1^2$ otherwise
        and we are done.

    \end{proof}

    We now show that any interpolation problem on some fixed finite set $X$ is solvable in terms 
    of fixed finite family of Gegenbauer polynomials.
    \begin{theorem} \label{finitegegentheorem}
        Let $X$ be a finite subset of $[-1,1]$ containing $1.$
        There exists an  $N>0$ depending on $X$ such that for
        any $f: X \rightarrow \BH$ such that 
            $\hat{f}(\gamma)= f(\gamma_{11})$
        a partially defined
        completely positive definite function on $SO(n)$ which extends to a continuous
        function completely positive definite function on $[-1,1]$ is in the cone generated by $P^{(n)}_0, \ldots, P^{(n)}_N.$
    \end{theorem}
    \begin{proof}
        Without loss of generality, $f(1)$ is the identity.
        Any such function with such a continuous extension must
        have continuous extension of the form
        $\sum A_iP^{(n)}_i(x)$ where $\sum A_i$ converges in the weak operator topology to the identity and each $A_i$ is positive semidefinite as for each vector $v \in \hilbert,$
        $\langle f v,v\rangle$ satisfies Theorem \ref{schoclassic}, and the minimal Stinespring representation arising in Theorem \ref{cpdthm} is unique.
        By Lemma \ref{hulllemma}, we may then rewrite each $P^{(n)}_j$ as a positive combination of 
        $P^{(n)}_0, \ldots, P^{(n)}_N$ to obtain a function that agrees at the interpolation points, so we are done.
    \end{proof}




    \section{Constrained angle codes and Delsarte bounds} \label{constrainedanglesection}
    Let $X \subseteq [-1,1).$ We define the {\bf Delsarte constant} $g_X$
    to be the maximum $\bar{g}\geq 0$ such that there exists
    $g: [-1,1]\rightarrow \mathbb{R}$ such that $g-\bar{g}$ positive definite on $S^{n-1}$
    such that $g|_X \leq 0$ and $g(1)=1.$
    We say $\mathcal{C}$ is a 
    {\bf constrained angle code with cosines in $X$} if all pairs of vectors in $\mathcal{C}$
    have the cosine of their angles in $X.$
    (Such codes were considered in \cite{DGS}.)

    We see the following via the same argument as the original Delsarte estimate.
    \begin{theorem} 
        Let $X \subseteq [-1,1).$ Let $\mathcal{C}$ be a constrained angle code with cosines
        in $X.$ Then,
            $|\mathcal{C}|\leq \frac{1}{g_X}.$
    \end{theorem}
    We leave the proof to the reader as it is exactly the same as the classical one presented above.
    (See proof of Theorem \ref{delsartegs}.)

    We say a set $X$ is {\bf Delsarte sharp} if there exists a constrained angle code such that $|\mathcal{C}| = \frac{1}{g_X}.$
    \begin{theorem}\label{firstsharptheorem}
        Let $X \subseteq (-1,1)$ be Delsarte sharp. Any function $g$
        such that $g-g_X$ is positive definite, $g|_X\leq 0$ and $g(1)=1$ is a polynomial.
    \end{theorem}
    \begin{proof}
        Take $X'$ to be the finite set corresponding to the optimal code.
        Note $g_{X'}=g_X.$
        Taking $g$ to be the optimal function for $X$
        it must have an expansion in a series of normalized Gegenbauer polynomials,
        but for large enough indices there exists an $\varepsilon>0$
        such that each $P^{(n)}_i/P^{(n)}_i(1)-\varepsilon$ positive definite on $X'$. 
        Each of these indices must have coefficient zero, or else 
        $g_{X'}>g_X$ as we could take the optimal candidate the for $X$ and restrict it to $X'.$
    \end{proof}

    Similarly we see the following result.
    We say $X$ is {\bf antipodal Delsarte sharp} if $X$ is Delsarte sharp with a code $\mathcal{C}$
    such that $\mathcal{C}=-\mathcal{C}.$
    \begin{theorem}
        Let $X \subseteq [-1,1)$ be antipodal Delsarte sharp such that
        $X\cap -X = X \cap (-1,1).$
        Any even function $g$
        such that $g-2g_X$ is positive definite, $g|_{X\setminus \{\pm 1\}}\leq 0$ and $g(1)=1$ is a polynomial.
    \end{theorem}

    For example, the codes of Sloane and Odlyzko correspond to Delsarte sharp codes \cite{odlyzko1988bounds}
    for their corresponding sets of angles. The difficulty in dimension three and four, necessitating sophisticated
    alternative arguments, as done by Musin \cite{musin2006kissing,musin2008kissing}, arises from a lack of Delsarte sharpness.
    Such is clearly analogous to the elegance and ease of sphere packing in dimension $8$ and $24$ \cite{viazovska2017sphere, cohn2017sphere} versus the onerous Herculean dimension 3 case chase by Hales \cite{hales2011sphere}.
    That low complexity algebraic functions in fact must be the witnesses in sharp Delsarte situations,
    and that the uniqueness or nonuniqueness can be obtained by a finite linear program, we speculate what the deep reason
    for the low algebraic complexity and (perhaps constrained uniqueness) of Viazovska's functions in the sphere packing regime
    \cite{viazovska2017sphere, cohn2017sphere,RV}. {Recently Kulikov, Nazarov, and Sodin \cite{Nazarov} found sufficient conditions for a pair of discrete subsets of the real line to be a uniqueness or a non-uniqueness pair for the Fourier transform. This has been generalized to $\mathbb{R}^{n}$ (although via different methods) in an unpublished work of Adve \cite{Adve}. }

\section{Partially defined complete positivity preservers} \label{partialpoz}
    \begin{definition}
        Let $X \subseteq \mathbb{R}.$
        Let $\hilbert$ be Hilbert space.
        We say that $f:X \rightarrow \BH$ is a partially defined complete postivity preserver
        if for any partially defined positive semi-definite matrix $[m_{ij}]_{ij},$
        we have that $[f(m_{ij})]_{ij}$ is again partially defined positive semidefinite.
    \end{definition}

    \begin{lemma}\label{statelemmapres}
        Let $X \subseteq [-1,1]$ such that $1\in X.$
        Let $\hilbert$ be Hilbert space.
        Let $f:X \rightarrow \BH$ is a partially defined complete positivity preserver.
        Let $G$ be a group with the discrete topology.
        Let $\varphi$ be a state on $\mathbb{C}[G]$.
        Let \[\Gamma = \{\gamma \in G| \varphi(\gamma) \in X\}.\]
        Then, $\hat{f}(g) = f(\varphi(g))$ is a partially defined completely positive definite
        function on $\Gamma.$
    \end{lemma}
    \begin{proof}
        By Lemma \ref{lemposcond}, it is sufficient to show that for any choice of
        $g_1, \ldots g_N$ that $[\hat{f}(g_i^*g_j)]_{i,j}$
        is a partially defined positive semidefinite matrix.
        Note that $[\varphi(g_i^*g_j)]$ is a partially defined positive semidefinite matrix, thus, as $f$ is a positivity preserver, we have that $[f(\varphi(g_i^*g_j))]=[\hat{f}(g_i^*g_j)]_{i,j}$ is as well.
    \end{proof}

    \begin{theorem}
        Let $X \subseteq [-1,1]$ such that $1\in X.$
        Let $\hilbert$ be Hilbert space.
        A function $f:X \rightarrow \BH$ is a partially defined complete postivity preserver
        if and only if
            \[f(x) = a\chi(x=1 \vee x=-1) + bx\chi(x=1 \vee x=-1) +  \sum c_ix^i\]
        where $a,b,c_i \geq 0, \sum c_i < \infty$ and $\chi$ is an indicator function. 
    \end{theorem}
    \begin{proof}
        Endow $SO(\infty)$ with the discrete topology.
        Let \[\Gamma = \{\gamma \in SO(\infty)| \gamma_{11} \in X.\}\]
        First, $f$ induces a completely positive definite function $\hat{f}$ on $\Gamma$ via
        $\hat{f}(\gamma) = f(\gamma_{11})$ by Lemma \ref{statelemmapres}
        By Theorem \ref{liftpd}, we have that $\hat{f}$ extends to a completely positive definite
        function on $SO(\infty)/SO(\infty)_+$ and is thus of the form Corollary \ref{compresthm} classifies the representations of type one that may occur, so we are done.
    \end{proof}

    If we do not contain the endpoints in a nice enough way or the domain is unbounded above, we get rid of the discontinuous extreme points simply by considering the function on smaller intervals.
    \begin{theorem}
        Let $X \subseteq \mathbb{R}$ such that either
        \begin{enumerate}
            \item $\sup X =\infty,$
            \item $\sup X>0, \sup X \geq -\inf X$ and $\pm\sup X \notin X.$ 
        \end{enumerate}
        Let $\hilbert$ be Hilbert space.
        A function $f:X \rightarrow \BH$ is a partially defined complete postivity preserver
        if and only if
            \[f(x) =  \sum c_ix^i\]
        where $c_i \geq 0, \sum c_i < \infty.$
    \end{theorem}
\section{Some concluding remarks} \label{conclusion}
    \subsection{Delsarte hallucinations }
        Finding particular finite sets $X$ that are Delsarte sharp is somewhat interesting. Our results give some finite linear program one can run to find the Delsarte constant for $X.$ Thus, one would hope that if you got an integer out, you would have a code. However, determining if there is a grammian constituted from particular entries is a difficult combinatorial matrix completion problem, and moreover, there exist {\bf Delsarte hallucinations}, sets $X$ such that $g_X$ reports an integer but there is no code with that size. For example, with $X=\{-1,1/2,-1/2,1\}, n= 10,$ we get a bound of $46,$ but the interested reader will find no such code.

        However, the method also produces genuine examples that are somewhat nontrivial.
        For example, taking $X=\{1/3,-1/3,1\}$ there is a spherical code with $10$ points as estimated for $n=5$ which 
        is obtained by labeling your code by vertices of the Petersen graph and putting a $1/3$ if there is no edge and $-1/3$ if there is an edge. The Petersen graph appears naturally in the theory of association schemes \cite{delsarte2002association}, so maybe this is not exceptionally surprising, but it is more exotic that say, some part of some super important well-known lattice.

  \subsection{Olshanski Spherical pairs}
    In the recent times, Olshanski made some new developments in the field of \textbf{asymptotic harmonic analysis} where he studied the inductive limits of homogeneous spaces as the dimension goes to infinity. More precisely, Olshanski \cite{Olshanskihowe} developed a general theory of the inductive limit $(G,K)$ of an increasing sequence $(G_{n}, K_{n})$ of Gelfand pairs--the so called spherical pairs\cite{Olshanskihowe}-- and studied the notion of spherical functions of such pairs. Moreover, he obtained some approximation results of the spherical functions of the inductive limit $(G,K)$ as limits of spherical functions for the Gelfand pairs $(G_{n},K_{n})$ in \cite{Olshanskiapprox} (such a sequence is called as the \textbf{Vershik--Kerov} sequence). We shall develop such results for the discontinuous representation of the ind-sphere $S^{\infty}$ and form some finitary conjectures. Although we shall deal with the sphere case (which is a rank one symmetric space), going through Olshanski's approach is proven to be very instructive. For more details see Faraut \cite{FarautTIFR}.
    
     \subsubsection{Gelfand Pairs}
     
     Let $G$ be a locally compact group , and $K$ a compact subgroup. The space $L^{1}(K\backslash G/K)$ of bi-$K$-invariant integrable functions on $G$ forms a convolution algebra. Recall that the pair $(G,K)$ is a \textbf{Gelfand pair} if $L^{1}(K\backslash G/K)$ is commutative. A nontrivial continuous function $\varphi$ is said to be \textbf{spherical} if 
     \[
     \int_{K} \varphi(xky)\ dk = \varphi(x) \varphi(y), \quad \text{for all}\ x,y\in G,
     \]
     where $dk$ is the normalized Haar measure on $K.$ Note that, a spherical function is bi-$K$-invariant and $\varphi(e) = 1.$

     Let $\mathcal{P}$ denote the continuous bi-$K$-invariant positive definite functions on $G$ and $\mathcal{P}_{1}$ denotes the convex base of functions $\varphi$ in $\mathcal{P}$ with $\varphi(e) = 1.$ The following theorem can be seen as the bi-invariant variant of the GNS:

     \begin{theorem}
        For a Gelfand pair $(G,K)$ and $\varphi \in \mathcal{P}_{1}$, the following are equivalent:
        \begin{enumerate}
            \item $\varphi$ is spherical.
            \item $\varphi \in \operatorname{Ext}(\mathcal{P}_{1})$, that is, $\varphi$ is an extreme point of the convex set $\mathcal{P}_{1}.$
            \item The GNS representation $\pi_{\varphi}$ associated with $\varphi$ is irreducible.
        \end{enumerate}
     \end{theorem}
     Let $\Omega$ denote the set of spherical functions of positive type and hence the extreme points of $\mathcal{P}_{1}.$ One remarkable fact is that, for a irreducible unitary representation $(\pi, \mathcal{H})$, the subspace of $K$-invariant vectors, call it $\mathcal{H}^{K}$ is of atmost dimension $1.$
     \begin{definition}
         We say that the unitary representation $(\pi, \mathcal{H})$ is of \textbf{class one} or \textbf{spherical} if the subspace $\mathcal{H}^{K}$ is nontrivial, that is, $\dim\mathcal{H}^{K}=1.$
     \end{definition}
     Hence the set $\Omega$ can also be seen as the set of equivalence classes of the class one representations of $\mathcal{H}.$ In this sense, the set of spherical functions of positive type $\Omega$ generalizes characters of a representation and therefore $\Omega$ is also referred as the \textbf{spherical dual} of the pair $(G,K).$ Note that, $\Omega$ is exactly the Gelfand spectrum for the Banach algebra $L^{1}(K\backslash G/ K),$ and hence it is locally compact under the topology of uniform convergence on compact subsets of $G$.
     \begin{theorem}[Bochner, Godement]
         For each $\varphi\in \mathcal{P},$ there exists a unique positive bounded measure $\mu$ on $\Omega$ such that
         \[
         \varphi(g) = \int_{\Omega} \chi(g)\ d\mu(\chi).
         \]
     \end{theorem}
     \subsubsection{Olshanski spherical pairs}
     For each $n\in \mathbb{N},$ let $(G_{n}, K_{n})$ be a Gelfand pair such that $G_{n}$ is a closed subgroup of $G_{n+1},\ K_{n}$ is a closed subgroup of $K_{n+1},$ and $K_{n} = G_{n} \cap K_{n+1}.$ Define
     \[
     G = \bigcup_{n=1}^{\infty} G_{n}, \quad K = \bigcup_{n=1}^{\infty} K_{n}.
     \]
     We can endow $G$ with the inductive limit topology. Then $K$ is a closed subgroup of $G.$ But, in general $G, K$ need not be locally compact and compact respectively. We call this $(G,K)$ to be an Olshanski spherical pair (OSP).
     Let $(G,K)$ be an OSP, as defined above. Now, let $\varphi: G\to \mathbb{C}$ be a continuous bi-$K$-invariant function. Then $\varphi$ is said to be \textbf{spherical} if for any $x,y\in G,$
     \[
     \lim_{n\to \infty} \int_{K_{n}} \varphi(xky)\ dk_{n} = \varphi(x)\varphi(y),
     \]
     where $dk_{n}$ be the normalized Haar measure on the compact group $K_{n}.$
      As in the case of a Gelfand pair, if $\varphi$ is a spherical function of positive type, there exists a class I representation $(\pi, \mathcal{H})$ of $G$(i.e irreducible, unitary, with $\dim \mathcal{H}^{K}  = 1$) such that 
      \[
      \varphi(x) = \langle \pi(x)v, v \rangle,\quad \text{where}\ v\in \mathcal{H}^{K} \setminus\{0\}.
      \]
      \subsubsection{Schoenberg's theorem through Olshanski's lenses}
      Let $G_{n} = SO(n+1)$ and $K_{n} = SO(n)$. By earlier discussion on $2$-point homogeneous spaces gives that $(G,K)$ is an OSP, where 
      \[
      G = \bigcup_{n=0}^{\infty} SO(n+1), \quad K = \bigcup_{n=0}^{\infty} SO(n)
      \]
      Note that, $K_{n}$ can be realised as a (compact) subgroup of $G_{n}$ by the identification that for any $A\in SO(n)$
\[
\begin{bmatrix}
    1 & 0 \\ 0 & A
\end{bmatrix} \in SO(n+1).
\]
     Let $\{e_{j}\}$ be the standard basis of $\mathbb{R}^{n+1}$, that is, $e_{j}(i) = \delta_{i,j}.$ By construction, $SO(n)$ fixes the first basis vector $e_{1}$. Now, let $\varphi: SO(\infty)\to \mathbb{C}$ be a continuous bi-$K$-invariant function. Then there exists a continuous $\psi: [-1,1]\to \mathbb{C}$ such that 
     \[
     \varphi(g) = \psi\big(\langle \pi(g) e_{1}, e_{1}\rangle \big)
     \]
     Suppose $x,y\in G$ with $\langle \pi(x) e_{1}, e_{1}\rangle = \cos\alpha$ and $\langle \pi(y) e_{1}, e_{1}\rangle = \cos\beta$. By Funk-Hecke formula,
     \[
     \int_{SO(n)} \varphi(xky)\ dk_{n} = \lambda_{n} \int_{0}^{\pi} \psi(\cos\alpha \cos\beta + \sin{\alpha}\sin{\beta}\cos{\theta}) \sin^{n-1}  {\theta}\ d\theta,
     \]
     where $\lambda_{n}$ is defined so that $\lambda_{n} sin^{n-1} \theta d\theta$ is a (sequence) of probability measures. In fact, $\lambda_{n} = \dfrac{\Gamma(\frac{n+1}{2})}{\sqrt{\pi} \Gamma(\frac{n}{2})}$. 
     \begin{lemma}[See Lemma 5.4 \cite{FarautTIFR}]\label{lem5.4}
         For any continuous function $f$ on $[0,\pi]$
         \[
         \lim_{n\to \infty} \lambda_{n} \int_{0}^{\pi} f(\theta) \sin^{n-1}\theta \ d\theta = f\bigg(\frac{\pi}{2}\bigg).
         \]
     \end{lemma}
     \begin{proof}
         This follows by the weak convergence of sequence of measures $\{\lambda_{n}\sin^{n-1}\theta \ d\theta\}$ to $\delta_{\pi/2}$
     \end{proof}
     Using this Lemma \ref{lem5.4},
     \[
     \int_{SO(n)} \varphi(xky)\ dk_{n} = \psi(\cos\alpha \cos\beta)
     \]
     Note that, if $\varphi$ is spherical then 
     \[
     \psi(\cos\alpha \cos\beta) = \psi(\cos\alpha) \psi(\cos\beta).
     \]
    So, the function $\psi$ is multiplicative on $[-1,1]$. Now a classical result of Banach/Sierpinski along with the regularity of the spherical functions imply that such a function has to be a power function, that is
     \[
     \varphi(g) = \langle \pi(g)e_{1}, e_{1} \rangle^{k},\quad \text{where}\ k\in \mathbb{N}.
     \]
     Therefore, the spherical dual is isomorphic to the naturals, that is, $\Omega \simeq \mathbb{N}.$ Thus, Krein-Milman theorem gives that any $\varphi\in \mathcal{P}_{1}$ can uniquely written as 
     \[
     \varphi(g) = \psi(\langle \pi(g)e_{1}, e_{1}\rangle),
     \]
     where
     \[
     \psi(t) = \sum_{k=0}^{\infty} a_{k} t^{k},
     \]
     with
     
     \[
     a_{k}\geq 0, \sum_{k = 0}^{\infty} a_{k} = 1.
     \]
     
     Therefore, any continuous bi-$K$-invariant positive definite function is of the above form, which in turn gives us the Schoenberg's theorem. A similar but non-trivial variant of this result for unitary groups $U(n)$ can be found in the works of Voiculescu \cite{Voiculescu} and Olshanski--Borodin \cite{BO}.

    \subsection{Finite dimension problems}
        \subsubsection{Somewhat comprehensible pathologies for $S^1$}
            Note that \[S^1 \cong \mathbb{Q}/\mathbb{Z} \oplus \bigoplus_{\mathbb{R}} \mathbb{Q}\] as $S_1 \cong \mathbb{R}/\mathbb{Z}.$
            Choosing such an isomorphism, and hence a basis for
            $\mathbb{R}$ as a vector space over $\mathbb{Q},$
            we see the following subgroups
            \[H_{\Lambda} = \left[\frac{1}{\lambda_\mathbb{Z}}\mathbb{Z}\right]/\mathbb{Z} \oplus \bigoplus_{r\in\mathbb{R}} \frac{1}{\lambda_r}\mathbb{Q}\]
            Now for any finite subset of $\mathbb{R}$
            the closure image of the generators $e_1, \ldots, e_r$
            for $H_{\Lambda}$ (where the indexing comes from the direct product structure)
            under the continuous homomorphisms $z^{kn}$ where $k=0,1,2 \ldots,$ and $n$ is fixed is all of
            $(S^1)^r.$ So, any homomorphism from $S^1$ to $S^1$
            is a pointwise limit of continuous homomorphisms.
            That is, the closure of the continuous homomorphisms
            in the topology of pointwise convergence is exactly
            the so-called Bohr compactification \cite{folland2016course}.
            In turn, $\overline{\mathbb{C}[S^1]}$ is isomorphic to $\overline{\mathbb{C}[S^1_{disc}]}$ where $S^1_{disc}$ is $S^1$ equipped with the nonstandard discrete topology, as the norm is given by the supremum of the norm over representations, and the noncontinuous ones can be approximated pointwise, so the norms on the dense set $\mathbb{C}[S^1]=\mathbb{C}[S^1_{disc}]$ agree.

            It is unclear to the authors whether or not
            $\overline{\mathbb{C}[SO(n)]}\cong \overline{\mathbb{C}[SO(n)_{disc}]}$ or for other general groups. Such creates some pathologies when trying to decide if a map is completely positive, as we lose the test condition given in Lemma \ref{lemmaposcondinit} for
            non-discrete groups.
            
            Namely, as Glimm type lemmas as in Theorem \ref{glimm} allow one to obtain any state as a limit of ones arising from continuous representations, if the two objects are isomorphic, there would be only two discontinuous class I representations of $SO(n)$ over $SO(n-1)$ for $n>3$ as Darboux type estimates as in Theorem \ref{darbouxest} imply the corresponding Gegenbauer polynomials go to $0$ on the interior of the unit interval.

         \subsubsection{Positive definite functions on finite spheres and consequences for representations of Class I of $SO(n)$}
            To classify the positive definite functions on finite spheres, we first classify the representations.
            By Dixmier's version of Glimm's lemma, given above as Theorem \ref{glimm}, we know that any pure state of $SO(n)$ which is invariant under $SO(n-1)$ is a weak$^*$ limit of a net of pure states. If we knew these states were limits of pure states which were $SO(n-1)$-invariant, then
            by Darboux type estimates as in Theorem \ref{darbouxest} give that any pathological
            limiting states would need to be the same as for $SO(n).$ Naive attempts to symmetrize approximating pure states fail as they will change the value of that identity,
            that is we cannot guarantee that
                \[\tilde{\sigma}(g) = \int_{SO(n-1)} \int_{SO(n-1)} \sigma(h_1^{-1}gh_2)dh_1,dh_2\]
            remains anything like unital or a good approximation.
            
            The following is a naive conjecture.
            \begin{conjecture}\label{compresthm3}
                Any irreducible class one representation of $SO(n)$
                arises from one of the following states:
                \begin{enumerate}
                    \item $\varphi(\gamma)=P^{(n-1)}(\gamma_{11})$
                    \item $\varphi(\gamma)=\chi(\gamma_{11}=1 \vee \gamma_{11}=-1)$ 
                    \item $\varphi(\gamma)=\gamma_{11}\chi(\gamma_{11}=1 \vee \gamma_{11}=-1)$ 
                \end{enumerate}
            \end{conjecture}
    
            We would thus immediately obtain the following conjectural corollary of Theorem \ref{cpdthm} which classifies positive definite functions on spheres viewed as homogeneous spaces given by quotients of the form $SO(n)/SO(n-1).$
            \begin{conjecture}\label{compresthm2}
                Let $n> 2$
                Let $f:[-1,1]\rightarrow\mathbb{R}$ such that 
                \begin{enumerate}
                    \item $f(1)=1,$
                    \item $\hat{f}(\gamma)=f(\gamma_{11})$
                    defines a positive definite function 
                    on $SO(n)/SO(n-1) \cong S^{n-1},$
                \end{enumerate}
                then $f$ is of the form
                    \[f(x) = a\chi(x=1 \vee x=-1) + bx\chi(x=1 \vee x=-1) + 
                    \sum c_iP^{(n-1)}(x)\]
                where $a,b,c_i \geq 0, a+b \sum c_i =1$ and $\chi$ is an indicator function.
            \end{conjecture}
            We note that given any sequence of pure states $\sigma_n$ corresponding to discontinuous
            representations of $SO(n)$ which are invariant under $SO(n-1),$
            that a limiting state on $SO(n)$ would have to come from Theorem \ref{compresthm}.

            Call a homogeneous space $G/H$ {\bf pathological} if there is a positive definite function that is not a pointwise limit of continuous ones. Essentially, our conjecture is that $SO(n)/SO(n-1)$ is not pathological, which we know to be true for
            $n=2,\infty.$ If such a quotient were pathological, one would see extra nontrivial estimates for finite constrained angle codes. However, such functions must be nonmeasurable \cite{dixmier}, which suggests that whether there is a pathological quotient might depend on set theoretic concerns. However, deciding if there is a particular constrained angle code is a finite problem, thus for any bound, there is a by hands proof by exhaustion of the bound.

    \bibliographystyle{plain}
    \bibliography{references}

\end{document}